\documentclass{article}

\usepackage{amsmath,amssymb,amsthm}
\usepackage{mathrsfs} % For math script font
\usepackage[all]{xy}

\usepackage[utf8]{inputenc} % For UTF-8 encoding
\usepackage[T1]{fontenc} % Font encoding

\usepackage{geometry}
\usepackage{xr-hyper}
\usepackage[colorlinks=
   citecolor=Black,
   linkcolor=Red,
   urlcolor=Blue]{hyperref}

\usepackage{color}

\newtheorem{thm}{Theorem}[section]
\newtheorem{lem}[thm]{Lemma}
\newtheorem{prop}[thm]{Proposition}
\newtheorem{conj}[thm]{Conjecture}
\newtheorem{cor}[thm]{Corollary}

\theoremstyle{definition}
\newtheorem{defn}[thm]{Definition}
\newtheorem{rmk}[thm]{Remark}

\makeatletter
\def\m{\@tut\@gobble\@tutt}
\def\@tut#1\@tutt#2{
\@ifnextchar,{\@tut #1 & #2\\\expandafter\expandafter\expandafter\@gobble\expandafter\@gobble\@gobble\@tutt}{
\@ifnextchar.{\begin{pmatrix}#1 &#2\end{pmatrix}\@gobble}{\@tut #1 & #2\@tutt}
}
}
\def\@tuut#1\@tuutt#2{
\@ifnextchar,{\@tuut #1 & #2\cr\expandafter\expandafter\expandafter\@gobble\expandafter\@gobble\@gobble\@tuutt}{
\@ifnextchar.{\bordermatrix{#1 &#2\cr}\@gobble}{\@tuut #1 & #2\@tuutt}
}
}

\makeatother

\newcommand{\ql}[1]{{{\textbf{Qirui}:}\color{blue}#1}} %% comment from Qirui
\newcommand\oldunderscore{_}

\newcommand{\GL}{\mathrm{GL}}         
\newcommand{\Mat}{\mathrm{Mat}}       
\newcommand{\End}{\mathrm{End}}       

\newcommand{\CO}{\mathcal{O}}         
\newcommand{\CF}{\mathcal{F}}         
\newcommand{\CG}{\mathcal{G}}         
\newcommand{\CM}{\mathcal{M}}         
\newcommand{\CN}{\mathcal{N}}         
\newcommand{\CH}{\mathcal{H}}
\newcommand{\CL}{\mathcal{L}}

\newcommand{\BC}{\mathbb{C}}
\newcommand{\BF}{\mathbb{F}}
\newcommand{\BZ}{\mathbb{Z}}
\newcommand{\Vol}{\mathrm{Vol}}
\newcommand\alpharef{\alpha^{\mathrm{ref}}}
\newcommand\betaref{\beta^{\mathrm{ref}}}

\newcommand\length{\mathrm{length}}
\newcommand{\lra}{\longrightarrow}    
\newcommand{\mb}[1]{\mathbf{#1}}      

\newcommand{\Orb}{\operatorname{Orb}} 
\newcommand{\Int}{\operatorname{Int}} 
\newcommand{\id}{\mathrm{id}}         

\newcommand{\lraiso}{\xrightarrow{\sim}}

\DeclareMathOperator{\tr}{tr}
\DeclareMathOperator{\Nm}{Nm}

\long\def\[#1]#2%
{%
  \begin{#1}#2\end{#1}%
}

\title{A proof for the biquadratic linear AFL for $GL_4$}
\author{Qirui Li}
\date{}
\begin{document}

\maketitle
\begin{abstract}
We prove both the biquadratic Guo--Jacquet Fundamental Lemma (FL) and the biquadratic linear Arithmetic Fundamental Lemma (AFL) for $\GL_4$ with the unit test function. Our approach relies on a detailed study of pairs of quadratic embeddings, which ultimately enables a reduction from the biquadratic case of $\GL_4$ to the coquadratic case of $\GL_2$. We further identify conditions under which the biquadratic case can be derived from the coquadratic case, and show that this reduction allow us to establish the conjectures for all orbits in $\GL_4$. As an additional consequence, we also prove the biquadratic FL for the identity test function in certain special families of orbits in $\GL_{2n}$. All results hold over both $p$-adic fields and local fields of positive characteristic.
\end{abstract}

\setcounter{tocdepth}{1}
\numberwithin{equation}{section}
\tableofcontents

\section{Introduction}

%\ql{The goal of this paper is to establish some relation between coquadratic case and biquadratic case.}
%\ql{This paper proves coquadratic linear AFL for unit test function for $\GL_{2n}$ implies the biquadratic linear AFL for identity test function of $\GL_{4n}$.}

\subsection{Context}

In \cite{Z12}, Zhang proposed a relative trace formula (RTF) approach to the arithmetic Gan--Gross--Prasad conjecture, which connects the first derivatives of certain L-functions to arithmetic intersection numbers. This method leads to local conjectures, most notably the arithmetic fundamental lemma (AFL) from \cite{Z12},\cite{Z13} and \cite{2305.14465}, and its variants in the presence of bad reduction \cite{1710.06962, 1906.12346}. These local statements assert identities between two quantities: derivatives of relative orbital integrals on the \textbf{analytic side}, and arithmetic intersection number on Rapoport--Zink (RZ) formal moduli spaces of $p$-divisible groups for unitary groups on the \textbf{geometric side}. By now, the AFL for the unit Hecke function as well as some of its variants in bad reduction have been proved in \cite{1909.02697, 2104.02779, 2112.11994,zhangzhiyuthesis,2108.02086,2305.14465}.

\vspace{1em}
\textbf{The linear AFL.} In his thesis \cite{L22}, the author introduced a linear version of the AFL conjecture. Here, \emph{linear} refers to the fact that both the relative orbital integrals and the RZ moduli spaces are defined in terms of general linear groups. The linear AFL conjecture can be understood as a first derivative variant of the Guo--Jacquet fundamental lemma \cite{Guo96}. The original Guo--Jacquet fundamental lemma, which does not involve taking derivatives of orbital integrals, was initially formulated as a higher--dimensional generalization of the relative trace formula approach used in the proof of the Waldspurger formula \cite{W}.

The main result of \cite{L22} established a closed, explicit, analytic formula for the intersection number side of the linear AFL. However, it is not clear how to identify its analytic side with the orbital integral derivative of the linear AFL, and so the linear AFL currently remains a conjecture. Still, the mentioned intersection number formula has been used to verify the linear AFL in low dimensions for the unit test function for $\GL_4$ and $\GL_2$ in \cite{1907.00090}, and also helped to obtain a new proof of a formula of Keating \cite{1902.10789}. It also gives an algorithm that enables computer verification of the linear AFL in certain special cases for general $\GL_{2n}$, as explored in the author's work \cite{1907.00090}.

\vspace{1em}
\textbf{The biquadratic linear AFL.} In his work with Howard \cite{2010.07365}, the authors extended the Guo--Jacquet fundamental lemma, as well as the linear AFL conjecture, to a biquadratic setting. Here, \emph{biquadratic} means that one considers two non-isomorphic quadratic extensions of the $p$-adic local field in question to define the relevant orbital integrals (resp. intersection numbers). This should be understood as a higher-dimensional (local) analog of the passage from the Gross--Zagier formula \cite{GZ} to the Gross--Kohnen--Zagier formula \cite{GKZ87}, see \cite{1707.00213} for the function field case. A key aspect here is that one allows some amount of ramification, while still preserving a fundamental lemma setting.

The main result of \cite{2010.07365} extends the intersection number formula from the author's thesis \cite{L22} to the biquadratic setting. Furthermore, \cite{2010.07365} provides evidence for the biquadratic FL and AFL by proving the cases of all Hecke functions for $\GL_2$.

\vspace{1em}
\textbf{Contributions of this paper.}
The present paper provides substantial further evidence for the biquadratic linear AFL in the $\GL_4$ case. More precisely, we prove the conjecture in full for the unit test function. Additionally, for $\GL_{2n}$, we also prove the conjecture for some special orbits with the unit test function.

We comment that, heuristically speaking and from a computational point of view, the biquadratic case of the linear AFL for $\GL_4$ is slightly simpler than the ``coquadratic'' case in \cite{1907.00090}. Loosely speaking, this is because non-isomorphic quadratic embeddings cannot be embedded very closely to each other into $M_4(F)$ which implies some amount of rigidity in the setting.

As explained in the introduction of \cite{2010.07365}, the broader motivation for the biquadratic linear AFL stems from global analogues such as the Gross--Zagier and Gross--Kohnen--Zagier formulas \cite{GZ, GKZ87}. Our results pave the way for future generalizations of these results, particularly towards extending the formulas of Howard--Shnidman \cite{1707.00213} to settings involving ramifications.

\vspace{1em}
\textbf{Further variants.}
The CM cycle arising from a ramified extension can be regarded as a mildly degenerate cycle. This perspective has further motivated the author’s ongoing work on a variant of the fundamental lemma in which one CM cycle is associated with the maximal order of an unramified quadratic extension, while the other corresponds to a more degenerate CM cycle whose endomorphism ring is merely a subring, rather than a maximal order. The corresponding intersection number in the $\GL_2$ case was computed in \cite{1902.10789}. In such cases, a perfect matching of Hecke functions is still expected.

A different type of ramification in the formalism of the linear AFL occurs in the presence of central simple algebras, see \cite{2307.11716, 2308.02458}. In the global context, this corresponds to a setting of twisted unitary groups. It would be very interesting to combine this variant with the biquadratic setting {\color{black} and the author hopes to return to this point in the future.} %\textbf{[Andreas: Please remove if you don't want this.]} }

%\vspace{1em}
%{\color{blue} \textbf{[Andreas: I suggest to remove this paragraph.]} \textbf{Function fields.}
%In their work \cite{1512.02683}, Yun--Zhang developed a generalization of the Gross--Zagier formula to study higher derivatives of L-functions for global function fields, and obtained complete, global results for $PGL_2$. Since their approach is not entirely based on the relative trace formula framework, they prove this result without knowing a higher derivative AFL. This has inspired ongoing work \cite{LMZ} on higher derivatives of orbital integrals, which, among other applications, enables a generalization of the character-twisted version of the Guo--Jacquet fundamental lemma in \cite{2209.08366} to its arithmetic form.}

\subsection{Pairs of quadratic embeddings (Double structures)}
The identities that make up the biquadratic linear AFL conjecture are parametrized by pairs of quadratic embeddings. Such pairs were first considered in the work of Howard--Shnidman on the Gross--Kohnen--Zagier formula for Heegner--Drinfeld cycles \cite{1707.00213}, and were further explored in the work of the author and Howard in \cite{2010.07365} for the local setting.

Let $K_1, K_2$ be two commutative rings with non-trivial involutions $\sigma_i: K_i \to K_i$ for $i = 1, 2$. Assume that their fixed points are isomorphic and identified, $F = (K_1)^{\sigma_1 = \id} = (K_2)^{\sigma_2 = \id}$. A pair of quadratic embeddings of an $F$-algebra $D$ is a pair of $F$-algebra homomorphisms $\alpha_1: K_1 \to D$, $\alpha_2: K_2 \to D$. We denote it as $\alpha = (\alpha_1, \alpha_2): (K_1, K_2) \to D$ for simplicity. A pair of quadratic embeddings is also called a double structure in \cite{1907.00090}.

In this paper, we study the coproduct $B_{K_1,K_2}$ of $K_1$ and $K_2$ in the category of (not necessarily commutative) $F$-algebras. If $K_1$ and $K_2$ satisfy certain conditions, see \eqref{c1} and \eqref{c2}, then $B_{K_1,K_2}$ is a quaternion algebra over $F[\mb w]$—the one-variable polynomial ring over $F$  (see Proposition \ref{generator}). A pair of quadratic embeddings $(K_1, K_2) \to D$ is then equivalent to a morphism of $F$-algebras $B_{K_1,K_2} \to D$. There are two canonical elements $\mb w, \mb z \in B_{K_1,K_2}$, where $\mb w$ commutes with all elements in $K_1$ and $K_2$, and $\mb z$ is a simultaneous semi-linear endomorphism in the sense that $\mb z x_i = x_i^{\sigma_i} \mb z$ for $x_i \in K_i$ and $i = 1, 2$. For a pair of embeddings $\alpha: (K_1, K_2) \to D$, we denote their images by $\mb w_\alpha$ and $\mb z_\alpha$.

\subsection{The linear biquadratic Fundamental Lemma (FL)}
Let $F$ be a non-Archimedean local field. The conditions \eqref{c1} and \eqref{c2} are satisfied when $K_1, K_2/F$ are quadratic étale extensions, with $K_1/F$ unramified. We fix a reference pair of quadratic embeddings $\alpha^{\mathrm{ref}}: (K_1, K_2) \to \Mat_{2h}(F)$, which gives rise to a pair of subgroups
$$\GL_h(K_1) \subset \GL_{2h}(F),\qquad \GL_h(K_2) \subset \GL_{2h}(F).$$
For any element $g \in \GL_{2h}(F)$, the conjugacy class of $\alpha_g := (g\alpha^{\mathrm{ref}}_1 g^{-1}, \alpha_2^{\mathrm{ref}})$ is identified with the set of orbits $\GL_h(K_2) \backslash \GL_{2h}(F) / \GL_h(K_1)$. In Definition \ref{rss}, we define the notion of a \textit{regular semisimple} pair $\alpha_g$. This notion precisely corresponds to that of regular semisimple orbits in $\GL_h(K_2) \backslash \GL_{2h}(F) / \GL_h(K_1)$.

Let $K_1/F$ be an unramified field extension and let $K_0\cong F\times F$. Let $K_3\subset K_1\otimes K_2$ be the fixed subalgebra of $\sigma_1\otimes\sigma_2$. Then there is an isomorphism $K_0\otimes K_1\cong K_1\otimes K_1$ and $K_3\otimes K_1\cong K_2\otimes K_1$. Two pairs of quadratic embeddings
$$
\beta:(K_0,K_3)\lra \Mat_{2h}(F), \qquad\alpha:(K_1,K_2)\lra \Mat_{2h}(F)
$$ 
are said to \textit{match} if there is an isomorphism $j:\Mat_{2h}(F)\otimes K_1\lra \Mat_{2h}(F)\otimes K_1$ of $K_1$-algebras such that the following diagram commutes
$$
\xymatrix{
(K_0,K_3)\otimes K_1\ar[rr]^\cong\ar[d]_{\beta}&& (K_1,K_2)\otimes K_1\ar[d]^\alpha\\
\Mat_{2h}(F)\otimes K_1\ar[rr]^j&& \Mat_{2h}(F)\otimes K_1.
}
$$
The notion of matching pairs gives rise to a correspondence of regular semisimple orbits
$$
\left(\GL_h(K_1)\backslash \GL_{2h}(F)/\GL_h(K_2)\right)^{r.s.s.}\lra
\left(\GL_h(K_0)\backslash \GL_{2h}(F)/\GL_h(K_3)\right)^{r.s.s.}.
$$ 

The Jacquet--Guo Fundamental Lemma is an identity comparing orbital integrals with bi-$\GL_{2h}(\CO_F)$-invariant test functions.

\begin{conj}[Generalization of Guo--Jacquet Fundamental Lemma]\label{FL}
Let $g_0\in\GL_{2h}(F)$ and $g_1\in\GL_{2h}(F)$ be two elements such that their orbits are regular semisimple and matching
$$\GL_h(K_0)\cdot g_0\cdot \GL_h(K_3)\longleftrightarrow \GL_h(K_1)\cdot g_1\cdot \GL_h(K_2).$$
Then we have an identity of orbital integrals for all bi-$\GL_{2h}(\CO_F)$-invariant functions $f$
$$
\Orb_{K_1,K_2}(f;g) = \Orb_{K_0,K_3}(f,g,0).
$$
\end{conj}
\begin{rmk}
The original Guo--Jacquet Fundamental Lemma is a statement for the case $K_0\cong K_3$ and $K_1\cong K_2$. The biquadratic generalization was conjectured in \cite{2010.07365}. To distinguish these two identities, we refer to them as the coquadratic and the biquadratic case. The coquadratic Guo--Jacquet Fundamental Lemma was proved for the characteristic function of $\GL_{2h}(\CO_F)$ for all $h$ by Guo in \cite{Guo96}.
\end{rmk}

%% Comments on Guo's method

\subsection{The linear biquadratic linear Arithmetic Fundamental Lemma (AFL)}

In fact, the matching of orbits 
$$
\left(\GL_h(K_1)\backslash \GL_{2h}(F)/\GL_h(K_2)\right)^{r.s.s.}\lra
\left(\GL_h(K_0)\backslash \GL_{2h}(F)/\GL_h(K_3)\right)^{r.s.s.}.
$$
is not surjective. Some of the missing orbits may appear in inner forms of $\GL_{2h}(F)$. Let $D$ be a division algebra of invariant $\frac1{2h}$ over $F$. Then $D\otimes F'\cong \Mat_{2h}(F')$ for any field extension $F'/F$ of degree $2h$. Let $\delta^{\mathrm{ref}}:(K_1,K_2)\lra D$ be a reference embedding of $K_1$ and $K_2$ into $D$ and let $D_{K_1}$ and $D_{K_2}$ be their centralizers. Then we have another setting of orbit matching
$$
\left(D_{K_1}^\times\backslash D^\times/D_{K_2}^\times\right)^{r.s.s.}\lra
\left(\GL_h(K_0)\backslash \GL_{2h}(F)/\GL_h(K_3)\right)^{r.s.s.}.
$$
Let $\gamma\in D^\times$ and $g\in\GL_{2h}(F)$ be a matching pair
$$
D_{K_1}^\times\cdot\gamma \cdot D_{K_2}^\times \longleftrightarrow \GL_h(K_0)\cdot g\cdot \GL_h(K_3).
$$
Then we have $\Orb(f,g,0)=0$ as a result of the functional equation from \cite{2010.07365}.

On the arithmetic-geometric side, consider $\CG_F$ a $1$-dimensional formal $\CO_F$-module over $\CO_{\breve F}$ of height $2h$. Then $\End(\CG_F) \cong \CO_{D}$, which is a maximal order in $D$. A pair of embeddings
$$
\delta: (K_1, K_2) \longrightarrow D
$$
gives rise to an embedding of maximal orders
\begin{equation}
\delta: (\CO_{K_1}, \CO_{K_2}) \longrightarrow \CO_{D} \cong \End(\CG_F).
\end{equation}
Let $\CM_{F}$ be the Lubin--Tate deformation space of $\CG_F$. Deforming $\CG_F$ with extra $\CO_{K_1}$- and $\CO_{K_2}$-actions via $\delta$ gives rise to an Artinian subscheme, which can be thought of as the intersection of two cycles $\CN_{K_1}$ and $\CN_{K_2}$, where $\CN_{K_i}$ is the closed subscheme deforming $\CG_F$ with the extra $\CO_{K_i}$-action. Denote the length of this intersection locus by $\Int(\delta)$. The biquadratic linear AFL for the identity test function is the following statement:
\begin{conj}\label{AFL}
We have
$$
\Int(\delta) = -\frac{1}{\ln q} \left. \frac{d}{ds} \right|_{s=0} \Orb(\mb 1, \beta, s).
$$
\end{conj}

\subsection{Main Results}
Our main result is now the following, proved in \S\ref{proofh=2calc}:

\begin{thm}\label{h=2calc}
The biquadratic Guo--Jacquet Fundamental Lemma and the biquadratic linear Arithmetic Fundamental Lemma hold for the characteristic function of $\GL_4(\CO_F)$.
\end{thm}

Moreover, we provide additional evidence for the correctness of the conjecture for the identity test function in higher dimensional cases, proved in \S\ref{prooftrivial}:

\begin{thm}\label{trivial}
The biquadratic Guo--Jacquet Fundamental Lemma holds for the identity test function for all orbits that satisfy $\mb z \in \GL_{2n}(\CO_F)$.
\end{thm}

We can also prove that some new cases for general $\GL_{2n}$ can be deduced from the (coquadratic) Guo--Jacquet FL and the linear AFL conjecture, in \S\ref{proofmaxred}:

\begin{thm}\label{maxred}
If $\CO_F[\mb w]$ is integral (i.e., $\CO_{F[\mb w]} = \CO_F[\mb w]$), then the (coquadratic) Guo--Jacquet FL for the unit test function and the linear AFL for the unit test function imply the biquadratic FL and the biquadratic linear AFL for that orbit with the unit test function.
\end{thm}

\subsection{Organization of Contents}
In Section \ref{pairs}, we discuss properties of pairs of quadratic embeddings, especially the elements $\mb w$ and $\mb z$, which play central roles throughout the paper.  
Section \ref{orbita} addresses orbital integrals. Using a concrete combinatorial interpretation, we prove that the biquadratic linear AFL holds for the identity test function when $\CO_F[\mb w]$ is integral. We also verify Theorem \ref{trivial}.  
In Section \ref{bfl}, we prove that the biquadratic Guo--Jacquet Fundamental Lemma holds for the characteristic function of $\GL_4(\CO_F)$.  
In Section \ref{bafl}, we prove that the biquadratic linear AFL holds for the characteristic function of $\GL_4(\CO_F)$.

\subsection{Acknowledgement} 
The results of this paper were originally obtained during the author's collaboration with Ben Howard. The author would like to thank him heartily for his encouragement and interest. The author also thanks Andreas Mihatsch for further encouragement and comments on a preliminary version.

%% biquadratic AFL is true for identity h=2

%% biquadratic FL is true for identity h=2, it is also true when the orbital integral does not depends on s.

%% For maximal generated case, the biquadratic FL and AFL for rank 4h is equivalent to co-quadratic linear linear AFL for rank 2h

\section{Pairs of quadratic embeddings}\label{pairs}
In this section, we study properties of the coproduct of two quadratic algebras. These structures parametrize the identities that make up the linear Arithmetic Fundamental Lemma and the Guo--Jacquet Fundamental Lemma.
\subsection{Quadratic ring extensions}
Let $F$ be a commutative ring and let $K_1\supset F$ and $K_2\supset F$ be two non-trivial ring extensions with two isomorphisms
$$
\sigma_1:K_1\lraiso K_1,\qquad
\sigma_2:K_2\lraiso K_2.
$$
such that 
$\sigma_1^2=\id_{K_1}$, $\sigma_2^2=\id_{K_2}$ and
$$
K_i^{\sigma_i=1}:=\{x\in K_i: x=x^{\sigma_i}\}=F
$$
for $i=1,2$. Furthermore, we require the existence of generators $\zeta_1\in K_1$ and $\varpi_2\in K_2$ satisfying 
\begin{flalign}
\label{c1}&\qquad\bullet (\zeta_1-\zeta_1^{\sigma_1})(\zeta_1+\zeta_1^{\sigma_1})\in K_1^\times;&&\\
\label{c2}&\qquad\bullet K_2=F[\varpi_2]=F\oplus F\varpi_2.
\end{flalign}
%{\color{blue} What if $K_2 = F\times K_2'$ where $K_2'/F$ is a quadratic field extension and $\sigma_2 = \mathrm{id}\times \sigma_2'$? My impression is you want something stronger like $K_2 = F\oplus F \varpi_2$.}

Note that the condition \eqref{c1} also implies $K_1=F[\zeta_1]$ since any $x\in K_1$ can be written into $x=a+b\zeta_1$ where
$b=(x-x^{\sigma_1})(\zeta-\zeta^{\sigma_1})^{-1}\in F$ and $a=x-b\zeta\in F$.
However, it may happens that $\zeta\notin K_1^\times$ --- an example is $K_1\cong F\times F$ and $\zeta_1=(1,0)$.
\begin{defn}\label{generator}
For a selected pair of generators $\zeta_1\in K_1$ and $\varpi_2\in K_2$, we define the \textit{intermediate generator}
\begin{equation}\label{defnw3}
\varpi_3=\zeta\otimes\varpi+\zeta^{\sigma_1}\otimes\varpi^{\sigma_2}\in K_1\otimes K_2, \qquad
\varpi_3^{\sigma_3}=\zeta^{\sigma_1}\otimes\varpi+\zeta\otimes\varpi^{\sigma_2}\in K_1\otimes K_2.
\end{equation}
\end{defn}
\begin{prop}
The intermediate generator satisfies 
\begin{equation}\label{overF}
\begin{cases}\varpi_3+\varpi_3^{\sigma_3}=(\zeta_1+\zeta_1^{\sigma_1})\otimes(\varpi_2+\varpi_2^{\sigma_2})&\in F
\\\\\varpi_3\cdot\varpi_3^{\sigma_3}=(\zeta_1^2+\zeta_1^{2\sigma_1})\otimes \varpi_2\varpi_2^{\sigma_2}
+
\zeta_1\zeta_1^{\sigma_1}\otimes (\varpi_2^2+\varpi_2^{2\sigma_2})&\in F.
\end{cases}
\end{equation}
\end{prop}
%and we have
%\begin{equation}\label{relationw3}
%\varpi'-\varpi'^{\sigma_3}=(\zeta-\zeta^{\sigma_1})(\varpi-\varpi^{\sigma_2}).
%\end{equation}
\subsection{Coproducts}

This subsection constructs the coproduct of $F$-algebras $B_{K_1,K_2} := K_1 \coprod K_2$ when $K_1$ and $K_2$ satisfy \eqref{c1} and \eqref{c2}. We introduce the canonical elements and detailed properties of $B_{K_1,K_2}$. Our method is to construct $B_{K_1,K_2}$ as an $F$-algebra and then show that it is isomorphic to the coproduct by verifying the universal property. Throughout this section, we fix our choice of generators $\zeta_1 \in K_1$ and $\varpi_2 \in K_2$.

%Let $\mathcal C_F$ be the category of $F$-algebras. Note that objects in $\mathcal C_F$ is not necessarily a commutative ring. Let $B$ be the coproduct of $K_1$ and $K_2$ in the category $\mathcal C_F$. In other words, $B$ is defined to be the $F$-algebra with two canonical embeddings 
%$$\alpha_1:K_1\lra B_K,\qquad \alpha_2:K_2\lra B_K$$
%suth that for any $D\in\CC_F$ and a pair of morphism of $F$-algebras $\beta_1:K_1\lra D$, $\beta_2:K_2\lra D$, there exists a unique morphism $\gamma:B_K\lra D$ such that $\beta_1=\gamma\circ\alpha_1$ and $\beta_2=\gamma\circ\alpha_2$. 

\begin{lem}\label{wzlemma}
Let $\mb w$, $\mb z$ be two elements in a non-commutative $K_1$-algebra such that
$$
\zeta_1\mb w=\mb w\zeta_1,\qquad
\zeta_1\mb z=\mb z\zeta_1^{\sigma_1},\qquad
\mb w\mb z=\mb z\mb w.
$$
Then for 
$$
x:=(\mb w+\mb z-\zeta_1^{\sigma_1}(\varpi_2+\varpi_2^{\sigma_2}))\cdot (\zeta_1-\zeta_1^{\sigma_1})^{-1},
$$
we have
\begin{equation}\label{xequation}
(x-\varpi_2)(x-\varpi_2^{\sigma_2})=\left((\mb w-\varpi_3)(\mb w-\varpi_3^{\sigma_3})-\mb z^2\right)(\zeta_1-\zeta_1^{\sigma_1})^{-2}.
\end{equation}
\end{lem}
\begin{proof}
By calculation,
$$
{\begin{split}
x-\varpi_2&=(\mathbf w+\mathbf z-(\zeta_1\varpi_2+\zeta_1^{\sigma_1}\varpi_2^{\sigma_2}))(\zeta_1-\zeta_1^{\sigma_1})^{-1}\\
&=(\mathbf w+\mathbf z-\varpi_3)(\zeta_1-\zeta_1^{\sigma_1})^{-1}.
\end{split}}
$$
Similarly,
$$
{\begin{split}
x-\varpi_2^{\sigma_2}&=(\mathbf w+\mathbf z-(\zeta_1^{\sigma_1}\varpi_2+\zeta_1\varpi_2^{\sigma_2}))(\zeta_1-\zeta_1^{\sigma_1})^{-1}\\
&=(\mathbf w+\mathbf z-\varpi_3^{\sigma_3})(\zeta_1-\zeta_1^{\sigma_1})^{-1}.
\end{split}}
$$
Since $\mb w$ commutes with $\zeta_1-\zeta_1^{\sigma_1}$ and $(\zeta_1-\zeta_1^{\sigma_1})\mb z=-\mb z(\zeta_1-\zeta_1^{\sigma_1})$, we have
$${
\begin{split}
(x-\varpi_2)(x-\varpi_2^{\sigma_2})&=(\mathbf w+\mathbf z-\varpi_3)(\zeta_1-\zeta_1^{\sigma_1})^{-1}(\mathbf w+\mathbf z-\varpi_3^{\sigma_3})(\zeta_1-\zeta_1^{\sigma_1})^{-1}\\
&=(\mathbf w-\varpi_3+\mathbf z)(\mathbf w-\varpi_3^{\sigma_3}-\mathbf z)(\zeta_1-\zeta_1^{\sigma_1})^{-2}.
\end{split}
}
$$
Since $(\mb w-\varpi_3)\mb z=\mb z(\mb w-\varpi_3^{\sigma_3})$, we have 
$$(\mathbf w-\varpi_3+\mathbf z)(\mathbf w-\varpi_3^{\sigma_3}-\mathbf z)=(\mathbf w-\varpi_3)(\mathbf w-\varpi_3^{\sigma_3})-\mb z^2$$
as desired.
\end{proof}
In the next proposition, we construct our proposed coproduct of $K_1$ and $K_2$, and prove that it is indeed the correct one.
\begin{prop}\label{generator}
Let $B_{K_1,K_2}$ be the free non-commutative algebra $K_1[\mathbf w,\mathbf z]$ modulo the following relations
\begin{enumerate}
\item \label{i1}$\mathbf w\cdot\zeta_1=\zeta_1\cdot\mathbf w$;
\item \label{i2}$\mathbf w\cdot\mathbf z=\mathbf z\cdot\mathbf w$;
\item \label{i3}$\mathbf z\cdot\zeta_1=\zeta_1^{\sigma_1}\cdot\mathbf z$;
\item \label{i4}$\left(\mathbf w-\varpi_3^{\sigma_3}\right)\left(\mathbf w-\varpi_3\right)=\mathbf z^2$. (see \eqref{defnw3} for definition of $\varpi_3$)
\end{enumerate}
The following assignment
\begin{equation}\label{bao}
\begin{array}{lll}
\beta_1: &K_1&\longrightarrow B_{K_1,K_2}\\\\
 &\zeta_1&\longmapsto\beta_1(\zeta_1):=\zeta_1
\end{array}
\qquad
\begin{array}{lll}
\beta_2: &K_2&\longrightarrow B_{K_1,K_2}\\\\
 &\varpi_2&\longmapsto
\beta_2(\varpi_2):=(\mathbf w+\mathbf z-\zeta_1^{\sigma_1}(\varpi_2+\varpi_2^{\sigma_2}))(\zeta_1-\zeta_1^{\sigma_1})^{-1}.
\end{array}
\end{equation}
give rise to a well-defined homomorphism of $F$-algebras with the property
\begin{equation}\label{w}
\mb w=\beta_1(\zeta_1)\beta_2(\varpi_2)+\beta_2(\varpi_2^{\sigma_2})\beta_1(\zeta_1^{\sigma_1});
\end{equation}
\begin{equation}\label{z}
\mb z=\beta_2(\varpi_2)\beta_1(\zeta_1)-\beta_1(\zeta_1)\beta_2(\varpi_2).
\end{equation}
Moreover, the pair 
$$K_1\overset{\beta_1}\longrightarrow B_{K_1,K_2}\overset{\beta_2}\longleftarrow K_2$$
is a universal pair in the sense that for any embedding $\alpha:(K_1,K_2)\lra D$, there is a unique $F$-algebra homomorphism $B_{K_1,K_2}\lra D$ such that the following diagram commutes
$$
\xymatrix{
B_{K_1,K_2}\ar@{-->}[rrd]&&\ar[ll]_{\beta_1} K_1\ar[d]^{\alpha_1}\\
K_2\ar[u]_{\beta_2}\ar[rr]^{\alpha_2}&&D.
}
$$
\end{prop}
\begin{proof}
Firstly, we need to show that $\beta_2: K_2 \to B_{K_1,K_2}$ is a well-defined ring homomorphism. For convenience, we verify this property by extending scalars from $F$ to $K_2$.
 By Lemma \ref{wzlemma}
$$
(\beta_2(\varpi_2)-\varpi_2)\cdot(\beta_2(\varpi_2)-\varpi_2^{\sigma_2})=((\mb w-\varpi_3)(\mb w-\varpi_3^{\sigma_3})-\mb z)(\zeta_1-\zeta_1^{\sigma_1})^{-2}=0,
$$
which implies that $\beta_2$ is a well-defined homomorphism of $F$-algebras.

With respect to the definition of $\beta_1(\zeta_1)$ and $\beta_2(\zeta_2)$, we first note that $\mb w$, $\zeta_1$, and $\varpi_2$ commute with $\beta_1(\zeta_1)$. Therefore
$$
\beta_2(\varpi_2)\beta_1(\zeta_1)-\beta_1(\zeta_1)\beta_2(\varpi_2)=(\mb z\zeta_1-\zeta_1\mb z)(\zeta_1-\zeta_1^{\sigma})^{-1}=\mb z,
$$
which proves \eqref{z}.
To obtain $\mb w$ from $\beta_1(\zeta_1)$ and $\beta_2(\varpi_2)$, we calculate
$$
\beta_1(\zeta_1)\cdot\beta_2(\varpi_2)=\zeta_1(\mb w+\mb z)(\zeta_1-\zeta_1^{\sigma_1})^{-1}-\frac{\zeta_1\zeta_1^{\sigma_1}(\varpi_2+\varpi_2^{\sigma_2})}{\zeta_1-\zeta_1^{\sigma_1}};
$$
and
$$
\beta_2(\varpi_2^{\sigma_2})\cdot\beta_1(\zeta_1^{\sigma_1})=-(\mb w+\mb z)\zeta_1^{\sigma_1}(\zeta_1-\zeta_1^{\sigma_1})^{-1}+\frac{\zeta_1\zeta_1^{\sigma_1}(\varpi_2+\varpi_2^{\sigma_2})}{\zeta_1-\zeta_1^{\sigma_1}}.
$$
Sum up these two equations and note that $\zeta_1\mb z-\mb z\zeta_1^{\sigma_1}=0$, we obtain the identity $\mb w=\beta_1(\zeta_1)\beta_2(\varpi_2)+\beta_2(\varpi_2^{\sigma_2})\beta_1(\zeta_1^{\sigma_1})$ as claimed in \eqref{w}.

For the last step, assume there are two embeddings $K_1\overset{\alpha_1}\lra D$ and $K_2\overset{\alpha_2}\lra D$. To prove the existence of an isomorphism $B_{K_1,K_2}\lra D$ mapping $\beta_1(\zeta_1)$ to $\alpha_1(\zeta_1)$ and $\beta_2(\varpi_2)$ to $\alpha_2(\varpi_2)$, we may first construct 
$$
\mb w'=\alpha_1(\zeta_1)\alpha_2(\varpi_2)+\alpha_2(\varpi_2^{\sigma_2})\alpha_1(\zeta_1^{\sigma_1})
$$
and
$$
\mb z'=\alpha_2(\varpi_2)\alpha_1(\zeta_1)-\alpha_1(\zeta_1)\alpha_2(\varpi_2).
$$
and then there is a morphism from free non-commuative algebra $K_1[\mb w,\mb z]\lra D$ such that $\mb w\longmapsto \mb w'$ and $\mb z\longmapsto \mb z'$. To induce a morphism from $B_{K_1,K_2}\lra D$, we need to prove the analogue relations $\mb w'\alpha_1(\zeta_1)=\alpha_1(\zeta_1)\mb w'$, $\mb z'\alpha_1(\zeta_1)=\alpha_1(\zeta_1^{\sigma_1})\mb z'$, $\mb w'\mb z'=\mb z'\mb w'$ and $(\mb w'-\varpi_3)(\mb w'-\varpi_3^{\sigma_3})=\mb z^2$. Note that our definition of $\mb w'$ and $\mb z'$ is the same as in \cite[(2.4.1),(2.4.2)]{2010.07365}. The proof of first three identities can be found in \cite[Prop.2.4.2]{2010.07365}. The last identity is also a result of \cite[Prop.2.4.2]{2010.07365} since 
$$\tr(\varpi_3)=\varpi_3+\varpi_3^{\sigma_3}=(\zeta_1+\zeta_1^{\sigma_1})(\varpi_2+\varpi_2^{\sigma_2})=\tr(\zeta_1)\cdot \tr(\varpi_2)$$
and
$$
\varpi_3\cdot\varpi_3^{\sigma_3}=\tr(\zeta_1^2)\Nm(\varpi_2)+\tr(\varpi_2^2)\Nm(\zeta_1).
$$
This completes the proof of this proposition.
\end{proof}

\subsection{The rings on analytic side}
In this section, we construct another ring $K_0$. We then show that $K_0 \cong F \oplus F$, and that $(K_0, K_3)$ is isomorphic to $(K_1, K_2)$ after base change to $K_1$. Furthermore, we prove that the choice of generators $\zeta_1 \in K_1$, $\varpi_2 \in K_2$ automatically defines a pair of generators $\zeta_0 \in K_0$, $\varpi_3 \in K_3$ satisfying \eqref{c1} and \eqref{c2}. Therefore, we may define their coproduct $B_{K_0,K_3}$, and it can be explicitly described in the same way as in Proposition~\ref{generator}. With respect to the generators $\zeta_0$ and $\varpi_3$, there are canonical elements $\mb w_{0,3}, \mb z_{0,3} \in B_{K_0,K_3}$. We denote the corresponding elements in $B_{K_1,K_2}$ by $\mb w_{1,2}$ and $\mb z_{1,2}$. The isomorphism $(K_0, K_3) \otimes K_1 \cong (K_1, K_2) \otimes K_1$ induces an isomorphism
$$
B_{K_0,K_3} \otimes K_1 \to B_{K_1,K_2} \otimes K_1.
$$
This section proves that $\mb w_{0,3}$ and $\mb z_{0,3}$ map to $\mb w_{1,2}$ and $\mb z_{1,2}$ under this isomorphism, respectively.

\begin{defn}
Let $K_0\subset K_1\otimes K_1$ be the subring fixed by $\sigma_1\otimes \sigma_1$. Let $\zeta_0$, $\zeta_0^{\sigma_0}$ be elements obtained by the following matrix product
$$
\m
{\zeta_0},
{\zeta_0^{\sigma_0}}.
=
\m
{\zeta_1\otimes 1}{\zeta_1^{\sigma_1}\otimes 1},
{\zeta_1^{\sigma_1}\otimes 1}{\zeta_1\otimes 1}.^{-1}
\m
{1\otimes\zeta_1},
{1\otimes \zeta_1^{\sigma}}..
$$
Explicitly,
$$
\zeta_0:=\frac{\zeta_1}{\zeta_1^2-\zeta_1^{\sigma_12}}\otimes \zeta_1-\frac{\zeta_1^{\sigma_1}}{\zeta_1^2-\zeta_1^{\sigma_12}}\otimes\zeta_1^{\sigma_1}.
$$
\end{defn}

\begin{prop}\label{K0isomorphism}
We have an isomorphism $K_0\cong F\oplus F$, and the non-trivial involution is given by $\sigma_0(a,b)=(b,a)$. And the image of $\zeta_0$ is $(1,0)$.
\end{prop}
\begin{proof}
Consider an isomorphism
$$
\iota:K_1\otimes K_1\lra K_1\oplus K_1,\qquad x\otimes y\longmapsto (xy,xy^\sigma).
$$
Denote by $(a,b):=(xy,xy^\sigma)$. The involution $x\otimes y\longmapsto x^{\sigma_1}\otimes y^{\sigma_1}$ induces $(a,b)\longmapsto (a^{\sigma_1},b^{\sigma_1})$ and $x\otimes y\longmapsto x^{\sigma_1}\otimes y$ induces $(a,b)\longmapsto (b,a)$. Therefore, $K_0$ is isomorphic to $K_1^{\sigma_1=\id}\oplus K_1^{\sigma_1=\id}=F\oplus F$ and $\sigma_0(a,b)=(b,a)$, which completes the proof of the proposition.

Next, we prove that the image of $\zeta_0$ is $(1,0)$. Indeed, under this isomorphism, we have mapped
$$
\zeta_0=\frac{\zeta_1}{\zeta_1^2-\zeta_1^{\sigma_12}}\otimes \zeta_1-\frac{\zeta_1^{\sigma_1}}{\zeta_1^2-\zeta_1^{\sigma_12}}\otimes\zeta_1^{\sigma_1}\longmapsto\left(\frac{\zeta_1^2-\zeta_1^{\sigma_12}}{\zeta_1^2-\zeta_1^{\sigma_12}},\frac{\zeta_1\zeta_1^{\sigma_1}-\zeta_1^{\sigma_1}\zeta_1}{\zeta_1^2-\zeta_1^{\sigma_12}}\right)=(1,0)
$$
as desired.
\end{proof}
The following definition defines $\varpi_3$ by the matrix form. But it is completely the same with our original definition in Definition \ref{generator}. We make the following definition for our convenient to write generators into matrix products. 
\begin{defn}
Let $K_3\subset K_1\otimes K_2$ be the subring fixed by $\sigma_1\otimes \sigma_2$. Let $\varpi_3$, $\varpi_3^{\sigma_3}$ be elements obtained by the following matrix product
$$
\m
{\varpi_3},
{\varpi_3^{\sigma_3}}.
=
\m
{\zeta_1\otimes 1}{\zeta_1^{\sigma_1}\otimes 1},
{\zeta_1^{\sigma_1}\otimes 1}{\zeta_1\otimes 1}.
\m
{1\otimes\varpi_2},
{1\otimes \varpi_2^{\sigma_2}}..
$$
Explicitly,
$$
\varpi_3:=\zeta_1\otimes\varpi_2+\zeta_1^{\sigma_1}\otimes\varpi_2^{\sigma_2}.
$$
\end{defn}

\subsection{Isomorphism of two pairs after base change}
\begin{defn}\label{def13}
Let $\iota:K_1\otimes K_3\lra K_1\otimes K_2$ be the homomorphism of $K_1$-algebras such that  
$$
\m
{\iota(1\otimes\varpi_3)},
{\iota(1\otimes\varpi_3^{\sigma_3})}.
=
\m
{\zeta_1\otimes 1}{\zeta_1^{\sigma_1}\otimes 1},
{\zeta_1^{\sigma_1}\otimes 1}{\zeta_1\otimes 1}.
\m
{1\otimes\varpi_2},
{1\otimes \varpi_2^{\sigma_2}}..
$$

\end{defn}
\begin{defn}\label{def10}
Abuse notation, also let $\iota:K_1\otimes K_1\lra K_1\otimes K_0$ be the homomorphism of $K_1$-algebras such that  
$$
\m
{\iota(1\otimes\zeta_1)},
{\iota(1\otimes\zeta_1^{\sigma_1})}.
=
\m
{\zeta_1\otimes 1}{\zeta_1^{\sigma_1}\otimes 1},
{\zeta_1^{\sigma_1}\otimes 1}{\zeta_1\otimes 1}.
\m
{1\otimes\zeta_0},
{1\otimes \zeta_0^{\sigma_0}}..
$$

\end{defn}
Since the coefficient matrix is invertible and all maps are $K_1$-linear, the inverse $\iota^{-1}:K_1\otimes K_0\lra K_1\otimes K_1$ and $\iota^{-1}:K_1\otimes K_2\lra K_1\otimes K_3$ are given by
$$
\m
{\iota^{-1}(1\otimes\varpi_2)},
{\iota^{-1}(1\otimes\varpi_2^{\sigma_2})}.
=
\m
{\zeta_1\otimes 1}{\zeta_1^{\sigma_1}\otimes 1},
{\zeta_1^{\sigma_1}\otimes 1}{\zeta_1\otimes 1}.^{-1}
\m
{1\otimes\varpi_3},
{1\otimes \varpi_3^{\sigma_3}}.,
$$
$$
\m
{\iota^{-1}(1\otimes\zeta_0)},
{\iota^{-1}(1\otimes\zeta_0^{\sigma_0})}.
=
\m
{\zeta_1\otimes 1}{\zeta_1^{\sigma_1}\otimes 1},
{\zeta_1^{\sigma_1}\otimes 1}{\zeta_1\otimes 1}.^{-1}
\m
{1\otimes\zeta_1},
{1\otimes \zeta_1^{\sigma_1}}..
$$

%$$
%\m
%{\frac{\zeta_1}{\zeta_1^2-\zeta_1^{2\sigma_1}}\otimes 1}{-\frac{\zeta_1^{\sigma_1}}{\zeta_1^2-\zeta_1^{2\sigma_1}\otimes 1}},
%{-\frac{\zeta_1^{\sigma_1}}{\zeta_1^2-\zeta_1^{2\sigma_1}\otimes 1}}{\frac{\zeta_1}{\zeta_1^2-\zeta_1^{2\sigma_1}\otimes 1}}.
%$$
\begin{prop}
The isomorphisms $K_1\otimes K_0\lraiso K_1\otimes K_1$ and $K_1\otimes K_3\lraiso K_1\otimes K_2$ induces a canonical isomorphism $K_1\otimes B_{K_0,K_3}\lraiso K_1\otimes B_{K_1,K_2}$. Denote corresponding elements by $\mb w_{0,3},\mb z_{0,3}\in B_{K_0,K_3}$ and $\mb w_{1,2},\mb z_{1,2}\in B_{K_1,K_2}$. Over generators this isomorphism can be explicitly written by
\begin{equation}\label{exdef}
\begin{array}{rcl}
c:K_1\otimes K_0[\mb w,\mb z]&\lra&K_1\otimes K_1[\mb w,\mb z]\\
\zeta_1\otimes 1&\longmapsto &\zeta_1\otimes 1\\
1\otimes \zeta_0&\longmapsto &\frac{\zeta_1}{\zeta_1^2-\zeta_1^{\sigma2}}\otimes \zeta_1-\frac{\zeta_1^{\sigma_1}}{\zeta_1^2-\zeta_1^{\sigma2}}\otimes\zeta_1^\sigma\\
\mb w_{0,3}&\longmapsto &\mb w_{1,2}\\
\mb z_{0,3}&\longmapsto &\mb z_{1,2}\\
\end{array}
\end{equation}
\end{prop}
\begin{proof}
Denote the canonical $F$-algebra embeddings by $(\alpha_1,\alpha_2): (K_1,K_2)\lra B_{K_1,K_2}$ and $(\alpha_0,\alpha_3): (K_0,K_3)\lra B_{K_0,K_3}$.
This statement is clear from definition. It suffices to show the commutativity of the following diagram with map defined in \eqref{exdef}
$$
\xymatrix{
K_1\otimes(K_0,K_3)\ar[rr]\ar[d]_{\id_{K_1}\otimes(\alpha_0,\alpha_3)}&& K_1\otimes(K_1,K_2)\ar[d]^{\id_{K_1}\otimes(\alpha_1,\alpha_2)}\\
K_1\otimes B_{K_0,K_3}\ar[rr]&&K_1\otimes B_{K_1,K_2}\\
}
$$
Firstly, the diagram 
$$
\xymatrix{
K_1\otimes K_0\ar[rr]\ar[d]_{\id_{K_1}\otimes \alpha_0}&& K_1\otimes K_1\ar[d]^{\id_{K_1}\otimes \alpha_1}\\
K_1\otimes B_{K_0,K_3}\ar[rr]&&K_1\otimes B_{K_1,K_2}\\
}
$$
commutes, since the horizontal maps defined by Definition \ref{def10} agree with the one defined in \eqref{exdef}. Therefore, it suffices to consider the commutativity of the following diagram
$$
\xymatrix{
K_1\otimes K_3\ar[rr]\ar[d]_{\id_{K_1}\otimes \alpha_3}&& K_1\otimes K_2\ar[d]^{\id_{K_1}\otimes \alpha_2}\\
K_1\otimes B_{K_0,K_3}\ar[rr]^c&&K_1\otimes B_{K_1,K_2}.\\
}
$$
Since all morphisms are $K_1$-linear, it suffices to check the following identity
\begin{equation}\label{check}
c\left(1\otimes\alpha_3(\varpi_3)\right) = \zeta_1\otimes\alpha_2(\varpi_2)+\zeta_1^{\sigma_1}\otimes\alpha_2(\varpi_2^{\sigma_2}).
\end{equation}
Using definition in \eqref{bao}, we have 
$$
\alpha_3(\varpi_3)=\left(\mb w+\mb z - \zeta_0^{\sigma_0}(\varpi_3+\varpi_3^{\sigma_3})\right)(\zeta_0-\zeta_0^{\sigma_0})^{-1}.
$$
By calculation, 
$$
{\[split]{
c(1\otimes(\zeta_0-\zeta_0^{\sigma_0})^{-1})&=  (\zeta_1-\zeta_1^{\sigma_1})\otimes   (\zeta_1-\zeta_1^{\sigma_1})^{-1}\\
&=\zeta_1\otimes   (\zeta_1-\zeta_1^{\sigma_1})^{-1} + \zeta_1^{\sigma_1}\otimes (\zeta_1^{\sigma_1}-\zeta_1)^{-1}
}}    
$$
and
$$
c(1\otimes\zeta_0^{\sigma_0}(\zeta_0-\zeta_0^{\sigma_0})^{-1}) = 
\frac1{\zeta_1+\zeta_1^{\sigma_1}}\cdot\left(\zeta_1\otimes \frac{\zeta_1^{\sigma_1}}{\zeta_1-\zeta_1^{\sigma_1}} 
-
\zeta_1^{\sigma_1}\otimes \frac{\zeta_1}{\zeta_1-\zeta_1^{\sigma_1}}\right).
$$
Furthermore, note that 
$$
\frac{\varpi_3+\varpi_3^{\sigma_3}}{\zeta_1+\zeta_1^{\sigma_1}}=\varpi_2+\varpi_2^{\sigma_2}
$$ 
Combining all the above equations, we have proved \eqref{check}. This establishes the desired commutativity of the diagram.
\end{proof}

%%% This is an invariant polynomial-free introduction
%%% When I introducing this orbital integral, is that necessary to introduce a map B_{0,3} to B_{1,2}? Freedom of z? Hilbert 90? What what? Yes, have to show that lattice have an isogeny among then. Hilbert 90 has to be used somehow. This would be used in the combinatorial interpretation. Regular semi simple is defined to be F[w] etale algebra. Then Hilbert 90 on etale algebra explain there is no relation to z. So counting sublattices naturally comes to counting z. And also has to modulo stablizer. 

\section{Orbital Integrals}\label{orbita}
From this section onward, let $F$ be a non-Archimedean local field, $K_1/F$ an unramified quadratic extension, and $K_2/F$ an arbitrary field extension. The involutions $\sigma_1$ and $\sigma_2$ are the unique non-trivial Galois conjugations. In this section, all embeddings $K_i \to \Mat_{2h}(F)$ are assumed to be free embeddings, in the sense that $F^{2h}$ is a free $K_i$-module of rank $h$. This assumption is automatic when $K_i \not\cong F \times F$. If $K_i \cong F \times F$, then $F^{2h}$ being a free module means that $(1,0) \cdot F^{2h} \cong F^h$ and $(0,1) \cdot F^{2h} \cong F^h$.

\subsection{Orbits}  
A pair of $F$-algebra embeddings $\beta: (K_1, K_2) \to \Mat_{2h}(F)$ gives rise to a group embedding  
$$
\left(\GL_h(K_1), \GL_h(K_2)\right) \hookrightarrow \GL_{2h}(F).
$$  
With respect to this embedding, the quotient $\GL_{2h}(F)/\GL_h(K_1)$ is a homogeneous space equipped with a left action of $\GL_{2h}(F)$ and a distinguished base point.
By restricting this action to the subgroup $\GL_h(K_2) \subset \GL_{2h}(F)$, the orbit of the base point is defined to be the orbit corresponding to the embedding data $\beta: (K_1, K_2) \to \Mat_{2h}(F)$.
The embedding $\beta: (K_1, K_2) \to \Mat_{2h}(F)$ is equivalent to a homomorphism
$$
B_{K_1, K_2} \to \Mat_{2h}(F)
$$
of $F$-algebras. Let $\Mat_h(K_1) \cong C(\beta_1) \subset \Mat_{2h}(F)$ be the centralizer of $K_1 \to \Mat_{2h}(F)$.

\begin{defn}
\label{rss}
We call the pair $\beta$ \emph{regular semisimple} if the image of $\mb w$ has distinct eigenvalues (over the algebraic closure) as an element in $C(\beta_1)$, and $\mb z$ is invertible.
\end{defn}

\subsection{Matching Orbits}

\begin{defn}
Two orbits corresponding to $(K_1,K_2) \to \Mat_{2h}(F)$ and $(K_0,K_3) \to \Mat_{2h}(F)$ are said to \emph{match} if there exists an isomorphism of $K_1$-algebras
$$
j : \Mat_{2h}(F) \otimes K_1 \to \Mat_{2h}(F) \otimes K_1
$$
such that the following two diagrams commute simultaneously:
$$
\xymatrix{
K_0 \otimes K_1 \ar[r] \ar[d] & K_1 \otimes K_1 \ar[d] \\
\Mat_{2h}(F) \otimes K_1 \ar[r]^j & \Mat_{2h}(F) \otimes K_1
}
\quad
\xymatrix{
K_3 \otimes K_1 \ar[r] \ar[d] & K_2 \otimes K_1 \ar[d] \\
\Mat_{2h}(F) \otimes K_1 \ar[r]^j & \Mat_{2h}(F) \otimes K_1
}
$$

This condition is equivalent to the commutativity of the following diagram:
$$
\xymatrix{
B_{K_0,K_3} \otimes K_1 \ar[rr]^c \ar[d]_{\beta} && B_{K_1,K_2} \otimes K_1 \ar[d]^{\alpha} \\
\Mat_{2h}(F) \otimes K_1 \ar[rr]^j && \Mat_{2h}(F) \otimes K_1
}
$$
where the upper horizontal map is defined in \eqref{exdef}.
\end{defn}

Let $\mb w_{0,3} \in B_{K_0,K_3}$ and $\mb w_{1,2} \in B_{K_1,K_2}$ denote the canonical elements. By \eqref{exdef}, we have $c(\mb w_{0,3}) = \mb w_{1,2}$. Since all automorphisms of matrix algebras are inner, the existence of $j$ implies that $\beta(\mb w_{0,3})$ must be conjugate to $\alpha(\mb w_{1,2})$. When the orbits are regular semisimple, the converse also holds; see \cite[Prop.~2.5.6]{2010.07365}.

%Since $\alpha(\mb w_{1,2})$ (resp. $\beta(\mb w_{0,3})$) commutes with elements in $\alpha_1(K_1)\subset \Mat_{2h}(F)$ (resp. $\beta_0(K_0)\subset \Mat_{2h}(F)$), the element $\alpha(\mb w_{1,2})$ is actually in centralizaer algebra of $\alpha_1(K_1)$ (resp. $\beta_0(K_0)$), which is isomorphic to $\Mat_h(K_1)$ (resp. $\Mat_h(K_0)$).
%restriction $j|_{K_1\otimes K}$ and $j|_{K_2\otimes K}$ induces the isomorphisms in \eqref{fixa}.

%% Also gives the combinatorial interpertation

\begin{defn}
Let $\CH_{2h} := \BC[\GL_{2h}(\CO_F) \backslash \GL_{2h}(F) / \GL_{2h}(\CO_F)]$ be the space of bi-$\GL_{2h}(\CO_F)$-invariant, compactly supported complex-valued functions on $\GL_{2h}(F)$. 
\end{defn}

The $\BC$-vector space $\CH_{2h}$ forms a $\BC$-algebra under convolution. By \cite[Appendix II]{LM22}, this algebra is generated by the elements $\{T_i^{\pm}\}_{i=0}^{2h}$, where $T_i(g) = 1$ if and only if $g \in \Mat_{2h}(\CO_F)$ and $\det g \in \pi^i \CO_F^\times$.

By \cite[Appendix II]{LM22}, any function in $\CH_{2h}$ can be expressed as a linear combination of convolution products of the form
$$
f = R_n * T_{m_1} * \cdots * T_{m_k},
$$
where $R_n$ is the characteristic function of $\pi^n \GL_{2h}(\CO_F)$, and $T_m$ is the characteristic function of the set
$$
\{ g \in \Mat_{2h}(\CO_F) : \det g \in \pi^m \CO_F^\times \}.
$$
This function admits a combinatorial interpretation. Let $\Lambda := \CO_F^{2h}$. For any $g \in \GL_{2h}(F)$, the value $f(g)$ equals the number of chains
$$
\left\{ \Lambda = \Lambda_0 \supset \Lambda_1 \supset \cdots \supset \Lambda_k = \pi^{-n} g \Lambda : \#(\Lambda_{i-1}/\Lambda_i) = q^{m_i} \right\}.
$$
This interpretation makes it evident that $f$ is bi-$\GL_{2h}(\CO_F)$-invariant.

\subsection{Orbital Integrals on the Geometric Side}
We begin by fixing a reference pair of quadratic embeddings
$$
\alpharef: (\CO_{K_1}, \CO_{K_2}) \longrightarrow \Mat_{2h}(\CO_F).
$$
Note that $\alpharef$ is not required to satisfy any special properties such as being regular semisimple. Let $\mathfrak{h}_1, \mathfrak{h}_2 \subset \Mat_{2h}(F)$ denote the centralizers of $\alpharef_1(K_1)$ and $\alpharef_2(K_2)$, respectively. Set $H_i := \mathfrak{h}_i \cap \GL_{2h}(F)$ for $i=1,2$.  We equip $H_1$ and $H_2$ with Haar measures normalized so that the compact open subgroups $H_i \cap \GL_{2h}(\CO_F)$ have volume 1.

%If $\alpha$ is regular semi-sinpke,  the subgroup $H_1\cap H_2$ is isomorphic to $F[\mb w]^\times$. We take the Haar measure on $H_1\cap H_2$ by the Haar measure on $F[\mb w]^\times$ normalized by $\CO_{F[\mb w]}^\times$.

Then for any pair of quadratic embeddings
$$
\alpha: (K_1, K_2) \longrightarrow \Mat_{2h}(F),
$$
there exist elements $g_1, g_2 \in \GL_{2h}(F)$ such that
$$
\alpha_1(\zeta_1) = g_1 \cdot \alpharef_1(\zeta_1) \cdot g_1^{-1}, \qquad
\alpha_2(\zeta_2) = g_2 \cdot \alpharef_2(\zeta_2) \cdot g_2^{-1}.
$$
Since $\CO_F^{2h}$ is stable under the action of $\alpharef(\CO_{K_1})$, the lattice $g_1 \cdot \CO_F^{2h}$ is preserved by $\alpha_1(\CO_{K_1})$. Similarly, $g_2 \cdot \CO_F^{2h}$ is stable under $\alpha_2(\CO_{K_2})$.

Define
$$
\CL_{\alpha_i} := \{ g_i \cdot h \cdot \CO_F^{2h} : h \in H_i \}
$$
for $i = 1, 2$. Then we can also write
\begin{equation}\label{latticeset}
\CL_{\alpha_i} = \left\{ \Lambda \subset F^{2h} : \alpha_i(\CO_{K_i}) \cdot \Lambda = \Lambda, \quad \Lambda \cong \CO_F^{2h} \right\},
\end{equation}
which is the set of all rank-$2h$ lattices preserved by the action of $\alpha_i(\CO_{K_i})$.

The set $\CL_{\alpha_i}$ carries a natural action of $L^\times$ for each $i = 1, 2$.

\subsubsection{Primitive Sublattices}

\begin{defn}\label{LKdef}
Let $\Gamma_{LK_1} \subset (LK_1)^\times$ be the subgroup such that, with $\Gamma_L := \Gamma_{LK_1} \cap L^\times$, we have
$$
(LK_1)^\times = \CO_{LK_1}^\times \cdot \Gamma_{LK_1}, \qquad
L^\times = \CO_L^\times \cdot \Gamma_L.
$$
\end{defn}

\begin{rmk}
If $K_1/F$ is an unramified field extension, then $K_1 \not\subset L$ whenever $h \in 2\BZ + 1$ or $L/F$ is ramified. In these cases, we have $\Gamma_{LK_1} = \Gamma_L$. Otherwise, $\Gamma_{LK_1}$ is strictly larger than $\Gamma_L$.
\end{rmk}

To define primitive sublattices, we must choose a splitting $\Gamma' \subset \Gamma_{LK_1}$ of the quotient
\begin{equation}\label{splitting}
\xymatrix{
\Gamma' \quad \ar@{^(->}[rr] \ar[rrd]_\cong && \Gamma_{LK_1} \ar@{->>}[d] \\
&& \Gamma_{LK_1} / \Gamma_L.
}
\end{equation}

\begin{defn}
A free rank-$2h$ $\CO_F$-submodule $\Lambda$ of $K_1L$ is called \emph{primitive} if
$$
\CO_{LK_1} \cdot \Lambda = \gamma \cdot \CO_{LK_1}
$$
for some $\gamma \in \Gamma'$.
\end{defn}

\begin{defn}\label{defn:primitive-lattices}
Fix a pair of embeddings $\alpha: (K_1, K_2) \to \Mat_{2h}(F)$. The vector space $F^{2h}$ is then equipped with a $K_1 \cdot L$-action, and we may choose an isomorphism $F^{2h} \cong K_1L$.
Choose a subgroup $\Gamma' \subset \Gamma_{LK_1} \subset (LK_1)^\times$ as in Definition~\ref{LKdef} and diagram~\eqref{splitting}. Define
$$
\CL_{\alpha_1}^\circ := \left\{ \Lambda \subset F^{2h} : \CO_{LK_1} \cdot \Lambda = \gamma \cdot \CO_{LK_1} \text{ for some } \gamma \in \Gamma' \right\}
$$
to be the set of primitive lattices.
\end{defn}

With respect to this definition, we have
\begin{equation}\label{uniformization}
\left(\CL_{\alpha_1}\times \CL_{\alpha_2}\right)/L^\times \cong \CL_{\alpha_1}^\circ\times\CL_{\alpha_2}.
\end{equation}

\subsubsection{Combinatorial Interpretation of Orbital Integrals}

\begin{defn}
For two rank-$2h$ submodules $\Lambda_1, \Lambda_2$, define $f(\Lambda_1, \Lambda_2) := f(g)$ such that $\Lambda_2 = g \Lambda_1$. This is well-defined because $f$ is bi-$\GL_{2h}(\CO_F)$-invariant.
\end{defn}

Let $f \in \BC[\GL_{2h}(\CO_F) \backslash \GL_{2h}(F) / \GL_{2h}(\CO_F)]$ be a spherical Hecke test function. When $K_1/F$ is an unramified quadratic extension, the orbital integral is defined by

\begin{equation}
\begin{split}
\Orb(f, \alpha) 
& := \int_{H_1 \times H_2 / L^\times} f(h_1^{-1} g_1^{-1} g_2 h_2) \, dh_1 \, dh_2 \\
& = \frac{1}{\Vol(\CO_L^\times)} \cdot \sum_{(\Lambda_1, \Lambda_2) \in \CL_{\alpha_1} \times \CL_{\alpha_2} / L^\times} f(\Lambda_1, \Lambda_2).
\end{split}
\end{equation}

We normalize the Haar measure on $L^\times$ so that $\Vol(\CO_L^\times) = 1$. Using the identification in \eqref{uniformization}, we obtain the following combinatorial expression for the orbital integral:

$$
\Orb(f, \alpha) = \sum_{\Lambda_1 \in \CL_{\alpha_1}^\circ} \sum_{\Lambda_2 \in \CL_{\alpha_2}} f(\Lambda_1, \Lambda_2).
$$

In particular, for $f = \mathbf{1}$ the characteristic function of $\GL_{2h}(\CO_F)$, we have

$$
\Orb(\mathbf{1}, \alpha) = \sum_{\Lambda \in \CL_{\alpha_1}^\circ \cap \CL_{\alpha_2}} 1 = \#\left( \CL_{\alpha_1}^\circ \cap \CL_{\alpha_2} \right).
$$

\subsection{Orbital Integrals for the Analytic Side}

The orbital integral for $(K_0, K_3)$ is defined similarly, but includes certain twisted characters. To formulate this precisely, we first introduce several structural components.

Let 
$$
\betaref : (\CO_{K_0}, \CO_{K_3}) \to \Mat_{2h}(\CO_F)
$$ 
be a reference embedding, where
$$
\betaref_0: K_0 \to \GL_{2h}(F); \qquad 
\zeta_0 \longmapsto \begin{pmatrix} I_h & 0 \\ 0 & 0 \end{pmatrix}, \qquad 
\zeta_0^{\sigma_0} \longmapsto \begin{pmatrix} 0 & 0 \\ 0 & I_h \end{pmatrix}.
$$

%%%%%%%%%%%5 Old fasioned writting following, banned Nov 14%%%%
%$$
%\begin{array}{rrcl}
%\betaref_0: &K_0&\lra& \GL_{2h}(F)\\
%& \zeta_0&\longmapsto&{\m{I_h}{0},00.}.
%\end{array}
%$$
%%%%%%%%%%%%%%%%%%%%%%%%%%%
Let $\mathfrak h_0$ and $\mathfrak h_3$ denote the centralizers of $\betaref_0$ and $\betaref_3$, respectively. Then $\mathfrak h_0$ is the set of $2h \times 2h$ matrices with block-diagonal structure, where each block is of size $h \times h$.

For any $\CO_{K_0}$-module $\Lambda$, define
$$
\Lambda_+ := \zeta_0 \cdot \Lambda, \qquad
\Lambda_- := \zeta_0^{\sigma_0} \cdot \Lambda.
$$
Since $\zeta_0^2 = \zeta_0$ and $\zeta_0 + \zeta_0^{\sigma_0} = 1$, we obtain a direct sum decomposition:
$$
\Lambda = \Lambda_+ \oplus \Lambda_-, \qquad 
\vec{v} = \underbrace{\zeta_0 \cdot \vec{v}}_{\in \Lambda_+} + \underbrace{\zeta_0^{\sigma_0} \cdot \vec{v}}_{\in \Lambda_-}.
$$

A morphism $\mb z$ of $\CO_F$-modules satisfying $\mb z \zeta_0 = \zeta_0^{\sigma_0} \mb z$ is equivalent to a pair of maps
$$
\mb z: \Lambda_+ \to \Lambda_-, \qquad
\mb z: \Lambda_- \to \Lambda_+.
$$

\subsubsection{Transfering Factors}

\begin{defn}
For any two $\CO_F$-submodules $\Lambda_1, \Lambda_2 \subset F^h$ of rank $h$, define
$$
[\Lambda_1 : \Lambda_2] := \frac{\#(\Lambda_1 / (\Lambda_1 \cap \Lambda_2))}{\#(\Lambda_2 / (\Lambda_1 \cap \Lambda_2))}.
$$
In particular, when $\Lambda_2 \subset \Lambda_1$, this simplifies to $[\Lambda_1 : \Lambda_2] = \#(\Lambda_1 / \Lambda_2)$.
\end{defn}

Let $\beta : (K_0, K_3) \to \GL_{2h}(F)$ be a pair of $F$-algebra embeddings. Choose elements $g_0, g_3 \in \GL_{2h}(F)$ such that
$$
\beta_0(\zeta_0) = g_0 \cdot \betaref_0(\zeta_0) \cdot g_0^{-1}, \qquad
\beta_3(\varpi_3) = g_3 \cdot \betaref_3(\varpi_3) \cdot g_3^{-1}.
$$
Then the lattices $\Lambda_0 := g_0 \cdot \CO_F^{2h}$ and 
$\Lambda_3 := g_3 \cdot \CO_F^{2h}$ 
are stable under the actions of 
$\beta_0(\CO_{K_0})$ and 
$\beta_3(\CO_{K_3})$, respectively.

\begin{defn}[Transferring Factor]\label{transferringfactor}
Let $\mb z_\beta$ be the semi-linear endomorphism associated with the embedding $\beta : (K_0, K_3) \to \GL_{2h}(F)$. Let $\Lambda$ be a lattice stable under $\beta_0(\CO_{K_0})$. Define the transferring factor by
$$
\Omega(\Lambda, s) := [\Lambda_{-} : \mb z \Lambda_{+}]^s \cdot (-1)^{\log_q [\Lambda_{-} : \mb z \Lambda_{+}]} 
= (-q^s)^{\log_q [\Lambda_{-} : \mb z \Lambda_{+}]}.
$$
\end{defn}

\begin{defn}
For any 
$$
h = \begin{pmatrix} h_+ & \\ & h_- \end{pmatrix},
$$
define $\Lambda'_0 := g_0 \cdot h \cdot \CO_F^{2h}$. Then the sublattices decompose as
$$
\Lambda'_{0+} = g_0 h_+ \cdot (\CO_F^{2h})_+, \qquad \Lambda'_{0-} = g_0 h_- \cdot (\CO_F^{2h})_-.
$$
It follows that
$$
[\Lambda'_{0-} : \mb z \Lambda'_{0+}] = [\Lambda_{0-} : \mb z \Lambda_{0+}] \cdot \frac{[\Lambda'_{0-} : \Lambda_{0-}]}{[\Lambda'_{0+} : \Lambda_{0+}]} 
= [\Lambda_{0-} : \mb z \Lambda_{0+}] \cdot \left| \frac{\det h_-}{\det h_+} \right|_F.
$$
Define two characters:
$$
|h| := \left| \frac{\det h_-}{\det h_+} \right|_F, \qquad 
\eta(h) := (-1)^{\log_q \left| \frac{\det h_-}{\det h_+} \right|_F}.
$$
\end{defn}

\begin{prop}
The transferring factor has the following properties:
\begin{enumerate}
\item For any $l \in L^\times$, we have $\Omega(l \Lambda, s) = \Omega(\Lambda, s)$.
\item The following identity holds:
\begin{equation}\label{character}
\Omega(g_0 \cdot h \cdot \CO_F^{2h}, s) = \Omega(g_0 \cdot \CO_F^{2h}, s) \cdot |h|^s \cdot \eta(h).
\end{equation}
\end{enumerate}
\end{prop}

\subsubsection{Orbital Integral for the Analytic Side and Combinatorial Interpretation}

Let 
$$
f \in \BC[\GL_{2h}(\CO_F) \backslash \GL_{2h}(F) / \GL_{2h}(\CO_F)]
$$ 
be a spherical Hecke test function. The (twisted) orbital integral for $(K_0, K_3)$ is defined by
$$
\Orb(f, \beta, s) := \int_{H_0 \times H_3 / L^\times} f(h_0^{-1} g_0^{-1} g_3 h_3) \cdot \Omega(g_0 h_0 \cdot \CO_F^{2h}, s) \, dh_0 \, dh_3.
$$
Using \eqref{character}, we may rewrite this in the following form:
\begin{equation}\label{splitorb}
\Orb(f, \beta, s) = \Omega(g_0 \cdot \CO_F^{2h}, s) \cdot 
\int_{H_0 \times H_3 / L^\times} f(h_0^{-1} g_0^{-1} g_3 h_3) \cdot |h_0|^s \cdot \eta(h_0) \, dh_0 \, dh_3,
\end{equation}
where $L^\times = g_0 H_0 g_0^{-1} \cap g_3 H_3 g_3^{-1}$.

Similarly, let $\CL_{\alpha_i}$ denote the set of rank-$2h$ submodules stable under $\alpha_i(\CO_{K_i})$, and let $\CL_{\beta_0}^\circ$ denote the set of primitive lattices. Using the same method of computation as in the previous section, we obtain the following combinatorial formula for the orbital integral:
\begin{equation}\label{comb-twist}
\Orb(f, \beta, s) = \sum_{(\Lambda_0, \Lambda_3) \in \CL^\circ_{\beta_0} \times \CL_{\beta_3}} 
f(\Lambda_0, \Lambda_3) \cdot \Omega(\Lambda_0, s).
\end{equation}

In particular, if $f = \mathbf{1}$ is the characteristic function of $\GL_{2h}(\CO_F)$, then
\begin{equation}\label{split-comb}
\Orb(\mathbf{1}, \beta, s) = \sum_{\Lambda \in \CL^\circ_{\beta_0} \cap \CL_{\beta_3}} 
\Omega(\Lambda, s) = \sum_{\Lambda \in \CL^\circ_{\beta_0} \cap \CL_{\beta_3}} 
(-q^s)^{\log_q [\Lambda_{-} : \mb z \cdot \Lambda_{+}]}.
\end{equation}

\subsection{Biquadratic Fundamental Lemma --- A special case when both $\mb w$ and $\mb z$ are units}\label{prooftrivial}

In this subsection, we provide partial evidence for the conjecture and prove one of the main result Theorem \ref{trivial}. Specifically, when $\mb z$ is a unit, the problem reduces to counting fixed lattices.

%We begin by proving a strengthened version of Hilbert's Theorem 90.
%
%\begin{lem}
%Let $R$ be a complete local ring such that $\frac{1}{2} \in R$, and let $\mathfrak{m} \subset R$ be the ideal of topologically nilpotent elements. Then for any $x \in 1 + \mathfrak{m}$, there exists an element $\sqrt{x} \in R$ such that $(\sqrt{x})^2 = x$.
%\end{lem}
%
%\begin{proof}
%Since $\frac{1}{2} \in R$, we can write the binomial expansion:
%$$
%(1 + x)^{\frac{1}{2}} = \sum_{n=0}^\infty \binom{1/2}{n} x^n \in R[[x]],
%$$
%which converges in $R$ and defines a square root of $x$. This completes the proof.
%\end{proof}
%
%We now recall the following conjecture—a biquadratic version of the Guo--Jacquet Fundamental Lemma—formulated in \cite{2010.07365}.

\begin{conj}\label{BFL}
Let $\alpha:(K_1, K_2) \to \Mat_{2h}(F)$ and $\beta:(K_0, K_3) \to \Mat_{2h}(F)$ be matching regular semisimple pairs. Then for any spherical Hecke function $f \in \CH_{2h}$, we have:
$$
\Orb(f, \alpha) = \Orb(f, \beta, 0).
$$
\end{conj}

In this section, we provide partial evidence for the conjecture in the case where $f = \mathbf{1}$.

\begin{thm}\label{trivialBFL}
Conjecture~\ref{BFL} holds when $f = \mathbf{1}$, $\alpha_1(\CO_{K_1}^\times), \alpha_2(\CO_{K_2}^\times) \subset \GL_{2h}(\CO_F)$, and $\mb w_\alpha \in \GL_{2h}(\CO_F)$.
\end{thm}

\begin{proof}
Let $\beta$ be the matching pair associated to $\alpha$. Then $\beta_1(\CO_{K_0}^\times), \beta_2(\CO_{K_3}^\times) \subset \GL_{2h}(\CO_F)$, and we may assume $\mb w_\beta = \mb w_\alpha =: \mb w$. Since $\mb w \in \GL_{2h}(\CO_F)$, we obtain
$$
\mb z^2 = (\mb w - \varpi_3)(\mb w - \varpi_3^{\sigma_3}) \in \GL_{2h}(\CO_F),
$$
which implies $\mb z \in \GL_{2h}(\CO_F)$.

A lattice $\Lambda \in \CL_{\beta_1} \cap \CL_{\beta_2}$ (respectively, $\Lambda \in \CL_{\alpha_1} \cap \CL_{\alpha_2}$) is stable under $\beta_1(\CO_{K_0})$ and $\beta_2(\CO_{K_3})$ (respectively, $\alpha_1(\CO_{K_1})$ and $\alpha_2(\CO_{K_2})$) if and only if it is preserved by $\beta_1(\CO_{K_0})$ (respectively, $\alpha_1(\CO_{K_1})$), $\mb w$, and $\mb z$.

Fix such a lattice $\Lambda$. The following argument applies symmetrically to either pair $(K_1, K_2)$ or $(K_0, K_3)$. Since $\mb z: \Lambda \to \Lambda$ is an automorphism, we have $\Omega(\Lambda, s) = 1$. Moreover, as $\Lambda$ is stable under both $\mb w$ and $\mb z$, it is also stable under their ratio $\mb z \cdot \mb w^{-1}$.

Let $R = \End(\Lambda, \alpha_1, \alpha_2)$ denote the ring of endomorphisms of $\Lambda$ commuting with both embeddings. Then $\mb z \cdot \mb w^{-1} \in \mathrm{Aut}(\Lambda)$. Since $(\mb z \cdot \mb w^{-1})^2 \in F[\mb w]$ and $\CO_F[\mb w] = \CO_L$, we conclude
$$
\frac{\mb z^2}{\mb w^2} \in \mathrm{Aut}(\Lambda) \cap F[\mb w] = \CO_L^\times.
$$

Let $\sigma_{KL/L}$ denote the nontrivial automorphism of $KL$ fixing $L$. Since $KL/L$ is unramified, we have $N_{KL/L}(LK^\times) \supset \CO_L^\times$. Therefore, there exists
\begin{equation}\label{mu}
\mu \in LK \cap \End(\Lambda) \quad \text{such that} \quad \mu \cdot \mu^{\sigma_{KL/L}} = \frac{\mb z^2}{\mb w^2}.
\end{equation}
Define $\sigma := \mu^{-1} \cdot (\mb z / \mb w)$. Then $\sigma$ is a $LK/L$-semilinear automorphism satisfying $\sigma^2 = 1$. Let
$$
M := (LK)^{\sigma = 1}
$$
be the $L$-subspace fixed by $\sigma$.

Let $\zeta \in \CO_K$ be such that $\CO_K = \CO_F[\zeta]$. We claim that
$$
\Lambda = (\Lambda \cap M) \oplus \zeta(\Lambda \cap M).
$$
It is known that $LK = M \oplus \zeta M$. We show that for any element $a + b\zeta \in \Lambda$ with $a,b \in M$, both $a$ and $b$ must lie in $\Lambda \cap M$.

Since $\Lambda$ is stable under $\sigma$, we have
$$
(a + b\zeta)^\sigma = a + b\zeta^{\sigma_{KL/L}} \in \Lambda.
$$
Therefore, the matrix
$$
\begin{pmatrix}
1 & \zeta^{\sigma_{KL/L}} \\
1 & \zeta
\end{pmatrix}\begin{pmatrix}a\\b\end{pmatrix} \in \Lambda \oplus \Lambda.
$$
The determinant of this matrix is $\zeta - \zeta^{\sigma_{KL/L}} \in \CO_{KL}^\times$, which implies it defines an isomorphism of $\Lambda \oplus \Lambda$. Hence both $a, b \in \Lambda$, as desired.

Since $M \cong L$, we conclude
$$
\Orb(1, \beta, 0) = \#\{ \Lambda \subset L : \CO_L \Lambda = \CO_L \}.
$$
This count is independent of the choice of $K_0$ or $K_3$, and the same reasoning applies to the pair $(K_1, K_2)$. Thus, both orbital integrals coincide, and the fundamental lemma holds in this case.
\end{proof}

Since $\mb w^2 \equiv \mb z^2$ modulo $\pi$ in biquadratic settings, we have proved Theorem \ref{trivial}.

\subsection{Reduction Formula}

Orbital integrals satisfy a reduction formula, which reduces their study to elliptic orbits.  Here, a regular semi-simple orbit is called \emph{elliptic} if its stabilizer is anisotropic modulo center. Equivalently, it is elliptic precisely if the étale $F$-algebra $F[\mb w]$ is a field. Our aim is to briefly recall this reduction formula. Recall that $f_n$ is the characteristic function of the set
$$
\{g \in \Mat_{2h}(\CO_F) : \det g \in \pi^n \CO_F^\times\},
$$
and let $R^n$ be the characteristic function of $\pi^n \cdot \GL_{2h}(\CO_F)$. Then the functions
$$
R^m \cdot f_{n_1} * \cdots * f_{n_k}
$$
span the Hecke algebra $\BC[\GL_{2h}(\CO_F) \backslash \GL_{2h}(F) / \GL_{2h}(\CO_F)]$ as a $\BC$-vector space.

\begin{defn}
We say that a pair $\alpha: (K_1, K_2) \to \Mat_{2h}(F)$ is \emph{hyperbolic} if there exist two pairs
$$
\alpha^{(0)}: (K_1, K_2) \to \Mat_{2h^{(0)}}(F), \quad \alpha^{(1)}: (K_1, K_2) \to \Mat_{2h^{(1)}}(F)
$$
such that there is a short exact sequence of $(K_1, K_2)$-modules
$$
0 \to F^{2h^{(0)}} \to F^{2h} \to F^{2h^{(1)}} \to 0.
$$
\end{defn}

The following theorem is due to \cite{LM22}.

\begin{thm}
For any $\underline{n} = (n_1, \dots, n_k) \in \BZ^k$, let $f_{\underline{n}} = f_{n_1} * \cdots * f_{n_k}$. Suppose there is an exact sequence
$$
0 \to (F^{2h^{(0)}}, \alpha^{(0)}) \to (F^{2h}, \alpha) \to (F^{2h^{(1)}}, \alpha^{(1)}) \to 0,
$$
then we have
$$
\Orb(f_{\underline{n}}, \alpha, s)
= R(\mb w^{(0)}, \mb w^{(1)}) \cdot \sum_{\underline{n}^{(0)} + \underline{n}^{(1)} = \underline{n}} \Orb(f_{\underline{n}^{(0)}}, \alpha^{(0)}, s) \cdot \Orb(f_{\underline{n}^{(1)}}, \alpha^{(1)}, s),
$$
where $R(\mb w^{(0)}, \mb w^{(1)})$ is an explicit rational factor.
\end{thm}

%{\color{blue} This proposition seems to be the definition of elliptic, no? I suggest to instead write the above sentence.

%\begin{prop}
%If $\alpha$ is elliptic, then $F[\mb w]$ is a field.
%\end{prop}}

\begin{cor}
If Conjecture~\ref{BFL} holds for all elliptic orbits, then it holds for all regular semisimple orbits.
\end{cor}
Therefore, it s sufficient to prove the conjectures for elliptic orbits.

%\subsection{Reduction to co-quadratic case}
%
%%%%%% This approach has very big problems although I kinda forgot what is the problem
%
%\begin{lem}
%If $L=F[\mathbf w]$ is a field,
%\begin{equation}\label{assumption1}
%(\mathbf w-\varpi_3)(\mathbf w-\varpi_3^{\sigma_3})=\mathbf z^2,
%\end{equation}
%and
%\begin{equation}\label{importantcondition}
%v_L(\mathbf z^2)>v_L((\varpi_3-\varpi_3^{\sigma_3})^2),
%\end{equation}
%then $F[\mathbf w,\mathbf z^2]=F[\varpi_3,\mathbf z^2]$.
%\end{lem}
%\begin{proof}
%It suffices to show $\varpi_3\in F[\mathbf w,\mathbf z^2]$ and $\mathbf w\in F[\varpi_3,\mathbf z^2]$. From \eqref{assumption1} and \eqref{importantcondition}, we have either 
%$$
%v_L(\mathbf w-\varpi_3)>v_L(\varpi_3^{\sigma_3}-\varpi_3)\qquad \text{ or } \qquad v_L(\mathbf w-\varpi_3^{\sigma_3})>v_L(\varpi_3^\sigma-\varpi_3).
%$$
%Without loss of generality, let us assume $v_L(\mathbf w-\varpi_3)>v_L(\varpi_3^{\sigma_3}-\varpi_3)\geq v_L(\varpi_3)$. Then 
%$$v_L(\mb w-\varpi_3^{\sigma_3})=v_L(\varpi_3^{\sigma_3}-\varpi_3)\qquad
%v_L(\mb w)=v_L(\varpi_3)
%$$
% by strong triangle inequality.
%
%\end{proof}

\section{Biquadratic Fundamental Lemma for $h=2$}\label{bfl}

The reason why the case $h=2$ of Conjectures~\ref{FL} and \ref{AFL} is amenable to direct calculation is that orders in quadratic extensions of $F$ have a particularly simple structure. Let $L$ be a quadratic étale $F$-algebra, and let $\pi \in F$ be a uniformizer.

Our goal in this section is to prove and apply Theorems~\ref{biquadraticFL} and \ref{orbitalh=2} to verify the biquadratic Guo--Jacquet Fundamental Lemma for the characteristic function of $\GL_4(\CO_F)$.

\begin{defn}
An \emph{$\CO_F$-order} $R \subset \CO_L$ is a subring containing $\CO_F$ such that $\CO_L / R$ has finite length as an $\CO_F$-module. For such an $R$, a \emph{proper fractional $R$-ideal} is an $\CO_F$-lattice $\mathfrak{a} \subset L$ such that
$$
R = \{ \lambda \in L : \lambda \mathfrak{a} \subset \mathfrak{a} \}.
$$
\end{defn}

\begin{prop}\label{prop:fractional class}
If $\Lambda \subset L$ is a free $\CO_F$-submodule of rank $2$, then
$$
\Lambda = x \cdot (\CO_F + \pi^n \CO_L)
$$
for some $x \in L^\times$ and some integer $n \geq 0$. Furthermore, if $\Lambda$ is a subring, then $\Lambda \cong \CO_F + \pi^n \CO_L$ for some $n \geq 0$.
\end{prop}

\begin{proof}
Let $m$ be the largest integer and $M$ the smallest integer such that
$$
\pi^m \CO_L \supset \Lambda \supset \pi^M \CO_L.
$$
Then
$$
\Lambda / \pi^M \CO_L \cong \CO_F / \pi^{m-M}.
$$
For any $x \in \Lambda \cap \pi^m \CO_L^\times$, we have
$$
\Lambda = \CO_F \cdot x + \pi^{M-m} \CO_L \cdot x = \left( \CO_F + \pi^{M-m} \CO_L \right) \cdot x,
$$
as claimed. If $\Lambda$ is a subring, then $1 \in x \cdot (\CO_F + \pi^n \CO_L)$ implies $x^{-1} \in (\CO_F + \pi^n \CO_L)$. Moreover, since $x \in \Lambda \subset \CO_L$ and $\Lambda$ is closed under multiplication, we conclude that $x \in \CO_L^\times$, and thus $x^{-1} \in (\CO_F + \pi^n \CO_L)^\times$ and
$$
\Lambda = x \cdot (\CO_F + \pi^n \CO_L) = x \cdot (\CO_F + \pi^n \CO_L) \cdot x^{-1} = \CO_F + \pi^n \CO_L.
$$
\end{proof}

From now on, define
$$
R_n := \CO_F + \pi^n \CO_L.
$$
Recall that a lattice $\Lambda \subset L$ is called \emph{primitive} if $\CO_L \cdot \Lambda = \CO_L$.

We introduce several definitions for later use.

\begin{defn}[Absolute values]
Suppose $L/F$ is an étale extension and let $\CO_L$ be the subring of topologically bounded elements. For any $x \in \CO_L$, define
$$
|x|_L := \#(\CO_L / x \CO_L).
$$
If $x \in LK$ for some finite extension $LK/L$ of degree $n$, define
$$
|x|_L := (|x|_{LK})^{1/n}.
$$
\end{defn}

\begin{defn}\label{index}
We write
$$
\Lambda_1 \overset{k}{\subset} \Lambda_2
$$
to indicate that $\Lambda_1 \subset \Lambda_2$ and $\#(\Lambda_2 / \Lambda_1) = q^k$.
\end{defn}

\begin{prop}\label{orbits}
For any $n \geq 1$, there are $q$ sublattices
$$
\Lambda \overset{1}{\subset} R_n
$$
such that $\Lambda \cong R_{n+1}$, and there is a unique lattice
$$
\Lambda \overset{1}{\subset} R_n
$$
such that $\Lambda \cong R_{n-1}$.
\end{prop}

\begin{proof}
For any $x \in R_n^\times$, let
$$
\Lambda = x \cdot R_{n+1} \subset R_n.
$$
Then
$$
\Lambda \overset{1}{\subset} R_n \quad \text{and} \quad \Lambda \cong R_{n+1}.
$$
Clearly,
$$
x \cdot R_{n+1} = y \cdot R_{n+1} \quad \iff \quad x y^{-1} \in R_{n+1}^\times.
$$
So there are exactly
$$
\#(R_n^\times / R_{n+1}^\times) = q
$$
such sublattices.

On the other hand,
$$
\pi \cdot R_{n-1} = \pi \CO_F + \pi^n \CO_L \overset{1}{\subset} R_n,
$$
so at least one lattice
$$
\Lambda \overset{1}{\subset} R_n \quad \text{with} \quad \Lambda \cong R_{n-1}
$$
exists, completing the proof.
\end{proof}

\subsection{Computation of Orbital Integrals on the Analytic Side for $GL\oldunderscore4$}

We consider the case where $L = F[\mb w]$ is a field.  

\begin{defn}[Conductor]\label{conductor}
Let $r \in \BZ$ be the integer such that 
$$
\CO_F[\mb w] = R_r := \CO_F + \pi^r \CO_L.
$$ 
\end{defn}

Hence, $\mb w \in R_n$ if and only if $n \leq r$.  
By \eqref{split-comb}, the orbital integral is given by
$$
\Orb(\mb 1, \beta, s) = \sum_{\Lambda \in \CL^\circ_{\beta_0} \cap \CL_{\beta_3}} (-q^s)^{\log_q [\Lambda_- : \mb z \cdot \Lambda_+]}.
$$

\begin{prop}
A lattice $\Lambda$ belongs to $\CL^\circ_{\beta_0} \cap \CL_{\beta_3}$ (respectively, $\CL^\circ_{\beta_1} \cap \CL_{\beta_2}$) if and only if:
\begin{itemize}
\item $\CO_{K_0} \cdot \Lambda = \Lambda$ (resp. $\CO_{K_1} \cdot \Lambda = \Lambda$);
\item $\mb z \cdot \Lambda \subset \Lambda$; 
\item $\mb w \cdot \Lambda \subset \Lambda$;
\item $\Lambda \cdot \CO_L = \gamma \cdot \CO_L$ for some $\gamma \in \Gamma'$.
\end{itemize}
\end{prop}

\begin{proof}
By Proposition~\ref{generator}, a lattice $\Lambda$ is stable under both $\beta_0(\CO_{K_0})$ and $\beta_3(\CO_{K_3})$ (resp. $\beta_1(\CO_{K_1})$ and $\beta_2(\CO_{K_2})$) if and only if it is stable under $\beta_0(\CO_{K_0})$ (resp. $\beta_1(\CO_{K_1})$), as well as under $\mb w$ and $\mb z$.
\end{proof}

Since $\CO_{K_0} \cdot \Lambda = \Lambda$, we can decompose 
$$
\Lambda = \Lambda_+ \oplus \Lambda_-.
$$ 
Then the stability conditions translate as:
$$
\mb z \cdot \Lambda \subset \Lambda \iff \mb z \cdot \Lambda_+ \subset \Lambda_-, \quad \mb z \cdot \Lambda_- \subset \Lambda_+,
$$
$$
\mb w \cdot \Lambda \subset \Lambda \iff \mb w \cdot \Lambda_+ \subset \Lambda_+, \quad \mb w \cdot \Lambda_- \subset \Lambda_-.
$$

We also fix $\Gamma' = \{ (1, \pi^\BZ) \}$. Then
$$
\CO_{K_0 L} \cdot \Lambda = \gamma \cdot \CO_{K_0 L} \text{ for some } \gamma \in \Gamma' 
\iff \CO_L \cdot \Lambda_- = \CO_L.
$$

Putting these together, we obtain the refined expression for the orbital integral:
$$
\Orb(\mb 1, \beta, s) = \sum_{\Lambda \in \CL^\circ_{\beta_0} \cap \CL_{\beta_3}} 
(-q^s)^{\log_q [\Lambda_- : \mb z \cdot \Lambda_+]} 
= \sum_{\substack{ \mb w \cdot \Lambda_- \subset \Lambda_- \\ \CO_L \cdot \Lambda_- = \CO_L }} 
\sum_{\Lambda_- \supset \mb z \cdot \Lambda_+ \supset \mb z^2 \cdot \Lambda_-} 
(-q^s)^{\log_q [\Lambda_- : \mb z \cdot \Lambda_+]}.
$$

\begin{prop}
The set
$$
\CL_L^{\mathrm{prim}} := \left\{ \CO_F^2 \cong \Lambda \subset L : \Lambda \cdot \CO_L = \CO_L \right\}
$$
admits an action of $\CO_L^\times$. Moreover, there is a bijection between the set of orbits and the set of non-negative integers:
$$
\CL_L^{\mathrm{prim}} \longrightarrow \BZ_{\geq 0}, \qquad \Lambda \cong R_n \longmapsto n.
$$
\end{prop}

\begin{proof}
This is equivalent to the following three claims:
\begin{enumerate}
\item If $\Lambda \cdot \CO_L = \CO_L$, then for any $l \in \CO_L^\times$, we have $l \cdot \Lambda \cdot \CO_L = \CO_L$;
\item If $\Lambda \subset L$ is an $\CO_F$-lattice with $\CO_F^2 \cong \Lambda$, then $\Lambda \cong R_n$ for some $n \geq 0$;
\item If $\Lambda_1, \Lambda_2 \cong R_n$ and both lie in $\CL_L^{\mathrm{prim}}$, then there exists $l \in \CO_L^\times$ such that $\Lambda_2 = l \cdot \Lambda_1$.
\end{enumerate}

Claim (1) is immediate. Claim (2) follows from Proposition~\ref{prop:fractional class}. For (3), note that Proposition~\ref{prop:fractional class} gives $\Lambda_2 = l \cdot \Lambda_1$ for some $l \in L^\times$. Since
$$
\CO_L = \CO_L \cdot \Lambda_2 = \CO_L \cdot (l \cdot \Lambda_1) = l \cdot \CO_L,
$$
we conclude $l \in \CO_L^\times$.
\end{proof}

We also note that $\mb w \in R_n$ if and only if $n \leq r$. Therefore, the orbital integral simplifies to
\begin{equation}\label{easyform}
\begin{aligned}
\Orb(\mb 1, \beta, s)
&= \sum_{\substack{ \mb w \cdot \Lambda_- \subset \Lambda_- \\ \Lambda_- \in \CL_L^{\mathrm{prim}} }} 
\sum_{\substack{ \Lambda_- \supset \mb z \cdot \Lambda_+ \supset \mb z^2 \cdot \Lambda_- \\ \mb w \cdot \Lambda_+ \subset \Lambda_+ }}
(-q^s)^{\log_q [\Lambda_- : \mb z \cdot \Lambda_+]} \\
&= \sum_{n=0}^{r} [\CO_L^\times : R_n^\times] \cdot 
\sum_{\substack{ R_n \supset \Lambda \supset \mb z^2 \cdot R_n \\ \mb w \cdot \Lambda \subset \Lambda }}
(-q^s)^{\log_q [R_n : \Lambda]}.
\end{aligned}
\end{equation}

This computation heavily relies on the properties of the elements $\mb w \in \CO_L$ and the semi-linear endomorphism $\mb z$. Recall that
$$
(\varpi_3 - \mb w)(\varpi_3^{\sigma} - \mb w) = \mb z^2.
$$
Since Theorem~\ref{trivialBFL} establishes the fundamental lemma when $\mb w \in \CO_L^\times$, we may assume $\mb w \in \CO_L \setminus \CO_L^\times$. Note that $|\mb w| > |\varpi_3|$ implies $\mb w \in \CO_L^\times$, so we are left with two cases:
\begin{itemize}
\item $|\mb w| = |\varpi_3|$, which implies that $L/F$ is ramified;
\item $|\mb w| < |\varpi_3|$.
\end{itemize}

\subsubsection*{Case: $|\mb w|_L < |\varpi_3|_L$}
In this case, we compute:
$$
|\mb z^2|_L = |\varpi_3 - \mb w|_L \cdot |\varpi_3^{\sigma_3} - \mb w|_L 
= |\varpi_3|_L \cdot |\varpi_3^{\sigma_3}|_L = |\pi|_L = q^2.
$$
Hence,
\begin{equation}\label{das}
\pi^{-1} \mb z^2 \in \CO_L^\times.
\end{equation}

\begin{lem}\label{pilemma}
If $|\mb w|_L < |\varpi_3|_L$, then for any $n \leq r$, we have
$$
\mb z^2 \cdot R_n = \pi R_n.
$$
\end{lem}

\begin{proof}
Since $|\mb z^2|_L = |\pi|_L$, we get $\pi^{-1} \mb z^2 \in \CO_L^\times$. Note that
$$
\mb z^2 \cdot R_n = \pi R_n \quad \Longleftrightarrow \quad \pi^{-1} \mb z^2 \in R_n^\times.
$$
It suffices to prove $\pi^{-1} \mb z^2 \in R_n$. Since $|\mb w|_L < |\varpi_3|_L$, we have $\mb w \in \pi \CO_L \cap R_r = \pi \cdot R_{r-1}$. Therefore,
$$
\mb z^2 = (\varpi_3 - \mb w)(\varpi_3^{\sigma} - \mb w) 
= \underbrace{\mb w^2}_{\in \pi^2 R_{r-1}} 
- \underbrace{\tr(\varpi_3) \cdot \mb w}_{\in \pi^2 R_{r-1}} 
+ \underbrace{\Nm(\varpi_3)}_{\in \CO_F}.
$$
Thus,
$$
\mb z^2 \in \CO_F + \pi^2 R_{r-1} = R_{r+1},
$$
and consequently,
$$
\mb z^2 \in \pi \CO_L \cap R_{r+1} = \pi R_r.
$$
Since $R_r \subset R_n$ for $n \leq r$, we conclude
$$
\pi^{-1} \mb z^2 \in R_r \subset R_n.
$$
\end{proof}

\begin{thm}\label{orbitalh=2}
If $L/F$ is an unramified extension, then
$$
\Orb(\mathbf{1}, \beta, s) = (1 + q^{2s}) + (1 - q^s)^2 \cdot (1 + q^{-1}) \cdot (q + q^2 + \cdots + q^r).
$$
In particular,
\begin{equation}\label{unramifiedcen}
\Orb(\mathbf{1}, \beta, 0) = 2.
\end{equation}

If $L/F$ is a ramified extension, then
$$
\Orb(\mathbf{1}, \beta, s) = (1 - q^s + q^{2s}) + (1 - q^s)^2 \cdot (q + q^2 + \cdots + q^r).
$$
In particular,
\begin{equation}\label{ramifiedcen}
\Orb(\mathbf{1}, \beta, 0) = 2.
\end{equation}
\end{thm}

\begin{proof}
Recall Definition~\ref{index}. We compute the orbital integral by collecting coefficients of $(-q^s)^k$. Since $|\mb z^2|_L = q^2$, we may write
$$
\Orb(\mathbf{1}, \beta, s) = a_0 + a_1(-q^s) + a_2(-q^s)^2,
$$
where
\begin{align*}
a_0 &= \sum_{n=0}^r \left[\CO_L^\times : R_n^\times\right] \cdot \#\left\{ \Lambda : R_n = \Lambda \supset \mb z^2 \cdot R_n \right\}, \\
a_1 &= \sum_{n=0}^r \left[\CO_L^\times : R_n^\times\right] \cdot \#\left\{ \Lambda : R_n \overset{1}{\supset} \Lambda \overset{1}{\supset} \mb z^2 \cdot R_n,\ \Lambda \cong R_k,\ k \leq r \right\}, \\
a_2 &= \sum_{n=0}^r \left[\CO_L^\times : R_n^\times\right] \cdot \#\left\{ \Lambda : R_n \supset \Lambda = \mb z^2 \cdot R_n \right\}.
\end{align*}

Clearly,
\begin{equation}\label{a02}
a_0 = a_2 = \sum_{n=0}^r \left[\CO_L^\times : R_n^\times\right].
\end{equation}

To compute $a_1$, we apply Lemma~\ref{pilemma}, and decompose
$$
a_1 = I_0 + \sum_{n=1}^r J_n + \sum_{n=1}^r K_n,
$$
where
\begin{align*}
I_0 &= \left[\CO_L^\times : \CO_L^\times\right] \cdot \#\left\{ \Lambda \cong \CO_L : \CO_L \overset{1}{\supset} \Lambda \overset{1}{\supset} \pi \CO_L \right\}, \\
J_n &= \left[\CO_L^\times : R_{n-1}^\times\right] \cdot \#\left\{ \Lambda \cong R_n : R_{n-1} \overset{1}{\supset} \Lambda \overset{1}{\supset} \pi R_{n-1} \right\}, \\
K_n &= \left[\CO_L^\times : R_n^\times\right] \cdot \#\left\{ \Lambda \cong R_{n-1} : R_n \overset{1}{\supset} \Lambda \overset{1}{\supset} \pi R_n \right\}.
\end{align*}

By Proposition~\ref{orbits}, we have:
$$
J_n = \left[\CO_L^\times : R_n^\times\right], \qquad
K_n = \left[\CO_L^\times : R_n^\times\right].
$$

For $I_0$, we distinguish cases:
$$
I_0 =
\begin{cases}
0 & \text{if $L/F$ is unramified}, \\
1 & \text{if $L/F$ is ramified}.
\end{cases}
$$

Thus:
\begin{itemize}
\item If $L/F$ is unramified:
$$
a_1 = 0 + 2 \sum_{n=1}^r \left[\CO_L^\times : R_n^\times\right] = 2(a_0 - 1);
$$
\item If $L/F$ is ramified:
$$
a_1 = 1 + 2 \sum_{n=1}^r \left[\CO_L^\times : R_n^\times\right] = 1 + 2(a_0 - 1).
$$
\end{itemize}

Hence we may write:
$$
a_1 = c_L + 2(a_0 - 1), \qquad \text{where} \quad
c_L =
\begin{cases}
0 & \text{if $L/F$ unramified}, \\
1 & \text{if $L/F$ ramified}.
\end{cases}
$$

Substituting into the orbital integral expression:
\begin{align*}
\Orb(\mathbf{1}, \beta, s)
&= a_0 + a_1(-q^s) + a_2(-q^s)^2 \\
&= (a_0 - 1)(1 - q^s)^2 + \left(1 + c_L \cdot (-q)^s + (-q^s)^2\right).
\end{align*}

To finish, we evaluate $a_0 - 1$ in each case:
\begin{itemize}
\item If $L/F$ is unramified:
$$
a_0 - 1 = \sum_{n=1}^r \left[\CO_L^\times : R_n^\times\right] = \sum_{n=1}^r (1 + q^{-1}) q^n;
$$
\item If $L/F$ is ramified:
$$
a_0 - 1 = \sum_{n=1}^r \left[\CO_L^\times : R_n^\times\right] = \sum_{n=1}^r q^n.
$$
\end{itemize}

This concludes the proof.
\end{proof}

\subsubsection{The case where $\mb w$ is a uniformizer}

\begin{thm}
If $v_L(\mb w) = 1$, then 
$$
\Orb(\mathbf{1}, \beta, s) = 1 - q^s + q^{2s} - q^{3s} + \cdots + (-q^s)^{v_L(\mb z^2)}.
$$
\end{thm}

\begin{proof}
If $v_L(\mb w) = 1$, then $\varpi_L := \mb w$ is a uniformizer of $\CO_L$, and $L/F$ is necessarily ramified. In this case, $\mb w \in R_n$ if and only if $n = 0$. 

The orbital integral \eqref{easyform} simplifies to
$$
\Orb(\mathbf{1}, \beta, s) = \sum_{\substack{\CO_L \supset \Lambda \supset \mb z^2 \CO_L \\ \Lambda \cong \CO_L}} (-q^s)^{\log_q[\CO_L : \Lambda]}.
$$

Since $\Lambda \cong \CO_L$ and $\Lambda \subset L$, such lattices are exactly of the form $\Lambda = \varpi_L^i \CO_L$ for $0 \le i \le v$, where $v := v_L(\mb z^2)$. Therefore,
$$
\Orb(\mathbf{1}, \beta, s) = \sum_{i = 0}^{v} (-q^s)^i,
$$
as claimed.
\end{proof}

\subsection{Computation of orbital integral on the geometric side}

To verify the Biquadratic Fundamental Lemma for $\GL_4$, the orbital integral equals the cardinality of the following set:
\begin{equation}\label{definelattice}
\CF := \left\{ \CO_{K_1}^2 \cong \Lambda \subset K_1 \otimes L : \mb z \Lambda \subset \Lambda,\quad \mb w \Lambda \subset \Lambda \right\} \big/ L^\times.
\end{equation}

\begin{lem}\label{non-empty}
The set $\CF$ is non-empty.
\end{lem}

\begin{proof}
A lattice $\Lambda \in \CF$ is stable under both $\beta_1(\CO_{K_1})$ and $\beta_2(\CO_{K_2})$. Since the initial embedding
$$
\beta : (\CO_{K_1}, \CO_{K_2}) \longrightarrow \GL_4(\CO_F)
$$
was chosen so that $\CO_F^4$ is fixed by both $\beta_1$ and $\beta_2$, the trivial lattice $\CO_F^4$ lies in $\CF$.
\end{proof}

\subsubsection{The case where $\mb w$ is topologically nilpotent but not a uniformizer}

In this part, we consider the case where $\mb w$ is topologically nilpotent (i.e., $\mb w \in \CO_F$ and $|\mb w| < 1$), but $\mb w$ is not a uniformizer of $\CO_F$.

\begin{lem}\label{downtoearth}
Let $L/F$ be a quadratic étale algebra and $\pi$ a uniformizer of $\CO_F$. If $x \in R_n$ satisfies $|x|_L = |\pi|_L^{1/2}$, then $n = 0$.
\end{lem}

\begin{proof}
Assume $n > 0$. Then
$$
R_n \subset \CO_F + \pi \cdot \CO_L = \CO_F^\times \cup \pi \cdot \CO_L.
$$
For $x \in \CO_F^\times$, we have $|x|_L = 1 \ne |\pi|_L^{1/2}$. For $x \in \pi \cdot \CO_L$, we have $|x|_L \leq |\pi|_L < |\pi|_L^{1/2}$, again a contradiction.
\end{proof}

\begin{lem}\label{nothingspecial}
Suppose $|\mb w|_L < |\varpi_3|_L$. If $\Lambda \subset K_1 \otimes L$ is a $\CO_{K_1}$-lattice with $\mb z \Lambda \subset \Lambda$, then
$$
\Lambda \cong \CO_{L'},
$$
where $L' := K_1 \otimes_F L$.
\end{lem}

\begin{proof}
Since $\pi^{-1} \mb z^2 \in \CO_L^\times$ by \eqref{das}, let $L' := K_1 \otimes_F L$. Suppose $\Lambda$ is a $\CO_{K_1}$-lattice of rank 2 with $\mb z \Lambda \subset \Lambda$. By Proposition~\ref{prop:fractional class}, we may write
$$
\Lambda \cong R_n' := \CO_{K_1} + \pi^n \CO_{L'}
$$
for some $n \geq 0$. Choose $\vec v \in \Lambda$ mapping to $1$ under this isomorphism. Then $\mb z \vec v = t \vec v$ for some $t \in R_n'$, and
$$
\mb z^2 \vec v = \mb z(t \vec v) = t^{\sigma_1} \mb z \vec v = t^{\sigma_1} t \vec v.
$$
Hence,
$$
\pi^{-1} t^{\sigma_1} t \in \CO_{L'}^\times,
$$
which implies $|t|_L = |\pi|_L^{1/2}$. By Lemma~\ref{downtoearth}, we must have $n = 0$, so $\Lambda \cong \CO_{L'}$.
\end{proof}

\begin{thm}\label{biquadraticFL}
Suppose $|\mb w|_L < |\varpi_3|_L$. Then:
\begin{itemize}
\item If $L/F$ is unramified, 
$$
\Orb(\mathbf{1}, \beta) = 2;
$$
\item If $L/F$ is ramified,
$$
\Orb(\mathbf{1}, \beta) = 1.
$$
\end{itemize}
In both cases, this confirms the Biquadratic Guo--Jacquet Fundamental Lemma for the characteristic function of $\GL_4(\CO_F)$, matching \eqref{unramifiedcen} and \eqref{ramifiedcen}.
\end{thm}

\begin{proof}
By Lemma~\ref{non-empty}, let $\Lambda$ be a $\CO_{K_1}$-lattice stable under both $\mb z$ and $\mb w$. Then $\Lambda \cong \CO_{L'}$ by Lemma~\ref{nothingspecial}.

If $L/F$ is ramified, then $L'$ is a field and $L'/L$ is unramified. Every such lattice is of the form $\varpi_L^n \cdot \CO_{L'}$ and hence belongs to a single $L^\times$-orbit. So,
$$
\Orb(\mathbf{1}, \beta) = 1.
$$

If $L/F$ is unramified, then $L \cong K_1$ and
$$
L \otimes_F K_1 \cong K_1 \oplus K_1.
$$
Let $\Lambda$ be a lattice stable under $\mb z$ and $\mb w$. Assume $\mb z \Lambda \cong (\pi, 1) \cdot \Lambda$ (the other case is symmetric). Any lattice $\Lambda' \cong \CO_{L'}$ can be written as
$$
\Lambda' = (\pi^{m_1}, \pi^{m_2}) \cdot \Lambda.
$$
Then
$$
\mb z \Lambda' = (\pi^{m_2+1}, \pi^{m_1}) \cdot \Lambda.
$$
The condition $\mb z \Lambda' \subset \Lambda'$ becomes
$$
\begin{cases}
m_2 + 1 \geq m_1, \\
m_1 \geq m_2,
\end{cases}
\quad \Rightarrow \quad m_1 = m_2 \text{ or } m_1 = m_2 + 1.
$$
So, there are exactly two $L^\times$-orbits:
$$
\Lambda, \quad (\pi, 1) \cdot \Lambda.
$$
Thus,
$$
\Orb(\mathbf{1}, \beta) = 2.
$$
\end{proof}

\section{Arithmetic Biquadratic Fundamental Lemma for $h = 2$}\label{bafl}

\subsection{Initial Settings}
Let $\CG_F$ be a $1$-dimensional formal $\CO_F$-module over $\CO_{\breve F}$ of height $2h$. Then $\End(\CG_F) \cong \CO_{D_F}$ is a maximal order in a division algebra $D_F$ of invariant $1 / (2h)$. A pair of embeddings
$$
\delta: (K_1, K_2) \longrightarrow D_F
$$
gives rise to an embedding of maximal orders
\begin{equation}\label{delta-pair}
\delta: (\CO_{K_1}, \CO_{K_2}) \longrightarrow \CO_{D_F} \cong \End(\CG_F),
\end{equation}
which equips $\CG_F$ with the structure of a $\CO_{K_1}$-module (denoted $\CG_{K_1}$) and a $\CO_{K_2}$-module (denoted $\CG_{K_2}$), each of height $h$.

Let $\CN_{K_1}$, $\CN_{K_2}$, and $\CM_F$ be the Lubin--Tate deformation spaces of $\CG_{K_1}$, $\CG_{K_2}$, and $\CG_F$, respectively. Then $\CN_{K_1}$ and $\CN_{K_2}$ are formal spectra of formal power series rings in $h-1$ variables over $\CO_{\breve K_1}$ and $\CO_{\breve K_2}$, respectively, while $\CM_F$ is defined over $\CO_{\breve F}$ with $2h - 1$ variables.

Including the base field dimension, we have:
$$
\dim(\CM_F) = 2h, \qquad \dim(\CN_{K_1}) = \dim(\CN_{K_2}) = h.
$$

Given the pair of embeddings in \eqref{delta-pair}, deforming $\CG_F$ with the additional $\CO_{K_i}$-structure via $\delta_i(\CO_{K_i}) \subset \End(\CG_F)$ yields two closed embeddings:
$$
\CN_{K_1} \longrightarrow \CM_F, \qquad \CN_{K_2} \longrightarrow \CM_F.
$$

These closed formal subschemes may be regarded as cycles of codimension $h$. One can show that if the pair $\delta = (\delta_1, \delta_2)$ is regular semisimple, then the intersection is Artinian. In this case, we define the intersection number:
$$
\Int(\delta) := \length_{\CO_{\breve F}} \left( \CN_{K_1} \times_{\CM_F} \CN_{K_2} \right).
$$

The following conjecture, proved in \cite{2010.07365}, is the arithmetic version of the biquadratic linear AFL:

\begin{conj}[Biquadratic Linear AFL for the Identity Test Function]
Let $\beta: (K_0, K_3) \to \Mat_{2h}(F)$ be a pair matching a regular semisimple pair $\delta: (K_1, K_2) \to \End(\CG_F) \otimes_{\CO_F} F$. Then
$$
\Int(\delta) = -\frac{1}{\ln q} \left. \frac{d}{ds} \right|_{s=0} \Orb(\mb 1, \beta, s).
$$
\end{conj}

Our main result in this subsection is the following:

\begin{thm}
The biquadratic linear AFL holds for $h = 2$.
\end{thm}

\begin{rmk}
The biquadratic linear AFL for $h = 1$ and arbitrary spherical Hecke functions was established in \cite{2010.07365}.
\end{rmk}

Our approach is to reduce the biquadratic case for $h = 2$ to the coquadratic case for $h = 1$, thus allowing us to deduce the result by known arguments.

\subsection{Maximal Order reduction}\label{proofmaxred}
Let $D$ be a central simple algebra over $F$. For a regular semi-simple pair $\alpha:(K_1,K_2)\lra D$, if $\CO_F[\mb w]=\CO_{F[\mb w]}$, then we have a reduction formalism which allows us to reduce the Fundamental Lemma and Arithmetic Fundamental Lemma from rank $2h$ to rank $h$ case. The main result of this subsection is Theorem \ref{orbitmat} and Lemma \ref{geometrymax}, which implies one of the main result Theorem \ref{maxred}.  The method depends on the following lemma. 

\begin{lem}\label{shiftpair}
 Let $B_{\CO_{K_1},\CO_{K_2}}$ be the coproduct of $\CO_{K_1}$ and $\CO_{K_2}$ in the category of $\CO_F$-algebras. Let $I\subset \CO_F[\mb w]$ be an ideal such that $\CO_F[\mb w]/I$ is integrally closed. Suppose $1+ \mb z\in \CO_D^\times$. Then the following assignment
$$
\widetilde \alpha: \CO_{K_1}\otimes_{\CO_F}\CO_F[\mb w]/I\lra B_{\CO_{K_1},\CO_{K_2}}/I;\qquad \zeta\longmapsto \widetilde\alpha(\zeta)
$$
where
$$
\widetilde\alpha(\zeta_1):=\left(1+\mb z\right)^{-1}(\zeta_1-\zeta_1^\sigma)+\zeta_1^{\sigma_1}
$$
extends to a morphism of $\CO_F$-algebras.
\end{lem}

\begin{proof}
Let 
$$\eta_1:=\left(1+\mb z\right)^{-1}\cdot(\zeta_1-\zeta_1^{\sigma})+\zeta_1^{\sigma},$$
$$\eta_1^{\sigma_1}:=\left(1+\mb z\right)^{-1}\cdot(\zeta_1^{\sigma}-\zeta_1)+\zeta_1.$$
We have
$$
\eta_1+\eta_1^{\sigma_1}=\zeta_1+\zeta_1^{\sigma_1}.
$$
 
Note that
$${\[split]{
\zeta_1^{\sigma_1}(1+z)^{-1}-(1+z)^{-1}\zeta_1&=(1+z)^{-1}\left((1+z)\zeta_1^{\sigma_1}-\zeta_1(1+z)\right)(1+z)^{-1}\\
&=(1+z)^{-1}(\zeta_1^{\sigma_1}-\zeta_1)(1+z)^{-1}.
}}
$$
Therefore 
$$
\eta_1\cdot\eta_1^{\sigma_1}=\left(\left(1+\mb z\right)^{-1}\cdot(\zeta_1-\zeta_1^{\sigma})+\zeta_1^{\sigma}\right)\cdot\left(\left(1+\mb z\right)^{-1}\cdot(\zeta_1^{\sigma}-\zeta_1)+\zeta_1\right)=\zeta_1\cdot\zeta_1^{\sigma_1}.
$$
Therefore, $\widetilde \alpha$ is a well-defined ring homomorphism.
\end{proof}

Recall that we denoted $L := F[\mb w]$. Therefore,
$$
\CO_{K_1} \otimes_{\CO_F} \CO_F[\mb w]/I = \CO_{K_1L}.
$$
Thus, the morphism $\widetilde{\alpha}$ in Lemma~\ref{shiftpair} is actually a map
$$
\widetilde{\alpha} : \CO_{K_1L} \longrightarrow B_{\CO_{K_1}, \CO_{K_2}} / I.
$$

\begin{defn}
Let $L \subset D$ be a subfield, and let $\alpha_1 : \CO_{K_1} \to D$ be a morphism of $\CO_F$-algebras such that the image of $\alpha_1$ centralizes $L$. We define the \emph{base change morphism} as
$$
\alpha_{1L} : \CO_{K_1L} \longrightarrow D.
$$
\end{defn}

Moreover, by Proposition~\ref{generator}, we have an isomorphism
$$
B_{\CO_{K_1}, \CO_{K_2}} \cong \CO_{K_1}[\mb w, \mb z] / (\text{relations}).
$$

\begin{lem}\label{revisedmat}
Let $D_\alpha$ and $D_\beta$ be central simple algebras over $F$, each equipped with a maximal order $\CO_{D_\alpha}$ and $\CO_{D_\beta}$ respectively. Suppose that
$$
D_\kappa \otimes_F F^{\mathrm{alg}} \cong \Mat_{2h}(F^{\mathrm{alg}})
$$
for $\kappa = \alpha, \beta$, where $F^{\mathrm{alg}}$ denotes the algebraic closure of $F$.

Let
$$
\alpha : (\CO_{K_1}, \CO_{K_2}) \longrightarrow \CO_{D_\alpha}, \qquad 
\beta : (\CO_{K_0}, \CO_{K_3}) \longrightarrow \CO_{D_\beta}
$$
be a pair of matching regular semisimple embeddings such that
$$
\CO_F[\mb w] = \CO_{F[\mb w]}.
$$
Let $\CO_{D_\alpha^L}$ and $\CO_{D_\beta^L}$ be the centralizers of $\CO_{F[\mb w]}$ in $\CO_{D_\alpha}$ and $\CO_{D_\beta}$ respectively. Assume that $1 + \mb z$ is invertible in both $\CO_{D_\alpha^L}$ and $\CO_{D_\beta^L}$.

Then the two pairs constructed in Lemma~\ref{shiftpair},
$$
(\alpha_{1L}, \widetilde{\alpha}) : (\CO_{K_1L}, \CO_{K_1L}) \longrightarrow \CO_{D_\alpha^L}, \qquad
(\beta_{1L}, \widetilde{\beta}) : (\CO_{K_0L}, \CO_{K_0L}) \longrightarrow \CO_{D_\beta^L},
$$
form a matching pair.
\end{lem}

\begin{proof}
Since $\alpha$ and $\beta$ form a matching pair, there exists an isomorphism over $K := K_1 \otimes K_3$
$$
j : D_\alpha \otimes_F K \xrightarrow{\sim} D_\beta \otimes_F K
$$
such that the following diagram commutes:
$$
\xymatrix{
B_{K_0, K_3} \otimes_F K_1 \ar[rr] \ar[d]_\beta && B_{K_1, K_2} \otimes_F K_1 \ar[d]^\alpha \\
D_\beta \otimes_F K_1 \ar[rr]^j && D_\alpha \otimes_F K_1.
}
$$
Therefore, $j$ maps $\mb z_\beta$ to $\mb z_\alpha$, and the image of $K_0 \otimes K_1$ to $K_1 \otimes K_1$. Since the construction of $\widetilde{\alpha}$ and $\widetilde{\beta}$ depends only on the element $\mb z$, we have
$$
\widetilde{\alpha} = \widetilde{\beta}.
$$
Hence, the pairs $(\alpha_{1L}, \widetilde{\alpha})$ and $(\beta_{1L}, \widetilde{\beta})$ form a matching pair.
\end{proof}

\begin{lem}\label{revisedmat2}
Let 
$$
\alpha : (\CO_{K_1}, \CO_{K_2}) \longrightarrow \CO_{D_\alpha}
$$ 
be a regular semisimple pair such that $\CO_F[\mb w] = \CO_{F[\mb w]}$ and $1 + \mb z$ is invertible in $\CO_{D_\alpha}$. Then the pair $(\alpha_{1L}, \widetilde{\alpha})$ constructed in Lemma~\ref{shiftpair} is also regular semisimple, with:
\begin{itemize}
    \item $\mb w_{\alpha_{1L}, \widetilde{\alpha}} = (1 - \mb z^2)^{-1}(\zeta_1^\sigma - \zeta_1)^2 + 2\zeta_1\zeta_1^\sigma$,
    \item $\mb z_{\alpha_{1L}, \widetilde{\alpha}} = -\mb z(1 - \mb z^2)^{-1}(\zeta_1 - \zeta_1^\sigma)^2$.
\end{itemize}
\end{lem}

\begin{proof}
By Definition~\ref{shiftpair}, we have:
\begin{align*}
\widetilde{\alpha}(\zeta_1) \cdot \alpha_1(\zeta_1) 
&= (1 + \mb z)^{-1}(\zeta_1 - \zeta_1^\sigma)\cdot \zeta_1 + \zeta_1 \zeta_1^\sigma, \\
\alpha_1(\zeta_1^\sigma) \cdot \widetilde{\alpha}(\zeta_1^\sigma) 
&= \zeta_1^\sigma \cdot (1 + \mb z)^{-1}(\zeta_1^\sigma - \zeta_1) + \zeta_1 \zeta_1^\sigma.
\end{align*}

Adding the two expressions, we obtain:
\begin{align*}
\mb w_{\alpha_{1L}, \widetilde{\alpha}} 
&= \widetilde{\alpha}(\zeta_1) \cdot \alpha_1(\zeta_1) + \alpha_1(\zeta_1^\sigma) \cdot \widetilde{\alpha}(\zeta_1^\sigma) \\
&= (1 + \mb z)^{-1}(\zeta_1 - \zeta_1^\sigma) \zeta_1 + \zeta_1 \zeta_1^\sigma 
+ \zeta_1^\sigma (1 + \mb z)^{-1}(\zeta_1^\sigma - \zeta_1) + \zeta_1 \zeta_1^\sigma \\
&= \left[(1 + \mb z)^{-1} \zeta_1 - \zeta_1^\sigma (1 + \mb z)^{-1}\right](\zeta_1 - \zeta_1^\sigma) + 2\zeta_1 \zeta_1^\sigma \\
&= \left[(1 - \mb z)\zeta_1 - \zeta_1^\sigma(1 - \mb z)\right](1 - \mb z^2)^{-1} (\zeta_1 - \zeta_1^\sigma) + 2\zeta_1 \zeta_1^\sigma \\
&= (1 - \mb z^2)^{-1} (\zeta_1 - \zeta_1^\sigma)^2 + 2\zeta_1 \zeta_1^\sigma,
\end{align*}
using the identities
$$
\zeta_1 \mb z = \mb z \zeta_1^\sigma, \qquad \zeta_1 (1 - \mb z^2)^{-1} = (1 - \mb z^2)^{-1} \zeta_1.
$$

Now we compute $\mb z_{\alpha_{1L}, \widetilde{\alpha}}$. From Definition~\ref{z}, we have:
\begin{align*}
\mb z_{\alpha_{1L}, \widetilde{\alpha}} 
&= \widetilde{\alpha}(\zeta_1) \cdot \alpha_1(\zeta_1) - \alpha_1(\zeta_1) \cdot \widetilde{\alpha}(\zeta_1) \\
&= (1 + \mb z)^{-1}(\zeta_1 - \zeta_1^\sigma) \zeta_1 - \zeta_1 (1 + \mb z)^{-1}(\zeta_1 - \zeta_1^\sigma) \\
&= \left[(1 + \mb z)^{-1} \zeta_1 - \zeta_1 (1 + \mb z)^{-1}\right](\zeta_1 - \zeta_1^\sigma) \\
&= \left[(1 - \mb z) \zeta_1 - \zeta_1 (1 - \mb z)\right] (1 - \mb z^2)^{-1} (\zeta_1 - \zeta_1^\sigma) \\
&= -\mb z (1 - \mb z^2)^{-1} (\zeta_1 - \zeta_1^\sigma)^2,
\end{align*}

Thus, $\mb w_{\alpha_{1L}, \widetilde{\alpha}}$ is a generator of $F[\mb w]$ and $\mb z_{\alpha_{1L}, \widetilde{\alpha}}$ is invertible, and hence the pair is regular semisimple.
\end{proof}

\subsubsection{Maximal order reduction of orbital integrals}

\begin{thm}\label{orbitmat}
Let 
$$
\beta : (K_0, K_3) \to \Mat_{2h}(F) \quad \text{and} \quad \alpha : (K_1, K_2) \to \Mat_{2h}(F)
$$ 
be regular semisimple pairs of quadratic embeddings such that 
$$
\CO_F[\mb w_\beta] = \CO_{F[\mb w_\beta]}, \quad \text{and} \quad \CO_F[\mb w_\alpha] = \CO_{F[\mb w_\alpha]},
$$ 
with $1 - \mb z_\beta$ invertible in $\CO_{D_\beta}$.

Then the orbital integrals satisfy:
\begin{align*}
\Orb(\mathbf{1}, (\beta_1, \beta_2), s) &= \Orb(\mathbf{1}, (\beta_{1L}, \widetilde{\beta}), s), \\
\Orb(\mathbf{1}, (\alpha_1, \alpha_2)) &= \Orb(\mathbf{1}, (\alpha_{1L}, \widetilde{\alpha})).
\end{align*}
\end{thm}

\begin{proof}
By the combinatorial definition, we have:
$$
\Orb(\mathbf{1}, (\beta_1, \beta_2), s) = \sum_{\Lambda \in \CL_{\beta_1}^\circ \cap \CL_{\beta_2}} \Omega(\Lambda, s).
$$

We analyze the intersection of lattice conditions:
\begin{align*}
\CL_{\beta_1}^\circ \cap \CL_{\beta_2}
&= \left\{ \Lambda \in \CL_{\beta_1}^\circ : \beta_2(\CO_{K_3}) \cdot \Lambda = \Lambda \right\} \\
&= \left\{ \Lambda \in \CL_{\beta_1}^\circ : (B_{\CO_{K_0}, \CO_{K_3}} / I) \cdot \Lambda = \Lambda \right\} \\
&= \left\{ \Lambda \in \CL_{\beta_{1L}}^\circ : \widetilde{\beta}(\CO_{K_0}) \cdot \Lambda = \Lambda \right\}.
\end{align*}

Here, $\CL_{\beta_{1L}}^\circ$ is defined analogously to Definition~\ref{defn:primitive-lattices}, and we have $\CL_{\beta_{1L}}^\circ = \CL_{\beta_1}^\circ$ as sets, since replacing $\beta_1$ with $\beta_{1L}$ does not alter $\Gamma'$, and lattices in $\CL_{\beta_{1L}}^\circ$ are automatically stable under the action of $\CO_L = \CO_F[\mb w]$.

By Definition~\ref{transferringfactor}, we write:
$$
\Omega(\Lambda, s) = (-q^s)^{\log_q [\Lambda_{-} : \mb z \Lambda_{+}]}.
$$

The base change $\beta_1 \leadsto \beta_{1L}$ preserves the decomposition $\Lambda = \Lambda_+ \oplus \Lambda_-$. Moreover, by Lemma~\ref{revisedmat2}, the new element 
$$
\mb z_{\beta_{1L}, \widetilde{\beta}} \in \mb z \cdot \CO_{D_\beta^L}^\times
$$ 
differs from the original $\mb z = \mb z_{\beta_1, \beta_2}$ only by a unit. Hence, the index 
$$
[\Lambda_{-} : \mb z \Lambda_{+}]
$$ 
remains unchanged, and so does $\Omega(\Lambda, s)$.

Therefore,
$$
\Orb(\mathbf{1}, (\beta_1, \beta_2), s) = \Orb(\mathbf{1}, (\beta_{1L}, \widetilde{\beta}), s).
$$

The proof for $\Orb(\mathbf{1}, (\alpha_1, \alpha_2)) = \Orb(\mathbf{1}, (\alpha_{1L}, \widetilde{\alpha}))$ proceeds identically, and is even simpler since no transfer factor $\Omega$ is involved.
\end{proof}

\subsubsection{Maximal Order Reduction of Intersection Numbers}

\begin{lem}\label{geometrymax}
Let $\delta : (K_1, K_2) \to D$ be a regular semisimple pair of quadratic embeddings such that $\CO_F[\mathbf{w}_\beta] = \CO_L$ and $1 - \mb z \in \CO_D^\times$. Let $\CG_L$ be the formal $\CO_L$-module of height $2$ obtained from the inclusion $\CO_L \subset \CO_D = \End(\CG_F)$. Then $\CO_{D_L} := \End(\CG_L)$ is the centralizer of $\CO_L$ in $\CO_D$. Let $\CM_{\CG_L}$ denote the deformation space of $\CG_L$. There is a canonical closed embedding
$$
\CM_{\CG_L} \hookrightarrow \CM_{\CG_F}.
$$

Let $\CN_{\delta_{1L}} \subset \CM_{\CG_L}$ and $\CN_{\widetilde\delta} \subset \CM_{\CG_L}$ be CM cycles obtained from the modified pair
$$
(\delta_{1L}, \widetilde\delta) : (\CO_{K_1L}, \CO_{K_1L}) \to \CO_{D_L}.
$$
Then we have an isomorphism of schemes:
$$
\CN_{\delta_{1L}} \cap \CN_{\widetilde\delta} \cong \CN_{\delta_1} \cap \CN_{\delta_2}.
$$
\end{lem}

\begin{proof}
The intersection $\CN_{\delta_{1L}} \cap \CN_{\widetilde\delta}$ parametrizes deformations of formal $\CO_L$-modules equipped with actions by both $\delta_{1L}(\CO_{K_1L})$ and $\widetilde\delta(\CO_{K_1L})$. By Lemma~\ref{shiftpair}, this data is equivalent to deforming a formal $\CO_F$-module with $\delta_1(\CO_{K_1})$ and $\delta_2(\CO_{K_2})$ actions. This completes the proof.
\end{proof}

The Theorem \ref{orbitmat} and Lemma \ref{geometrymax} implies Theorem \ref{maxred}.

\subsection{Proof of the biquadratic linear AFL for $\GL\oldunderscore4$}\label{proofh=2calc}

\begin{thm}
The biquadratic linear AFL holds for the characteristic function of $\GL_4(\CO_F)$.
\end{thm}

\begin{proof}
Let $\CG_F$ be a formal $\CO_F$-module of height $4$ over $\overline{\BF}_q$, and let
$$
\delta : (\CO_{K_1}, \CO_{K_2}) \longrightarrow \End(\CG_F)
$$
be a pair of $\CO_F$-embeddings. Let $\mb w_\delta$, $\mb z_\delta$ be the corresponding central element and semilinear endomorphism associated with $\delta$. Let $\varpi_D$ denote the uniformizer of $\CO_D := \End(\CG_F)$.

Since $\End(\CG_F)$ is a quaternion algebra and $\mb z_\delta \cdot \delta(\zeta_1) = \delta(\zeta_1^\sigma) \cdot \mb z_\delta$, we must have
$$
\mb z_\delta \in \varpi^{2\BZ + 1} \cdot \CO_D^\times.
$$
If $\mb z_\delta \in \varpi \cdot \CO_D^\times$, then $\delta_1(\zeta_1)$ and $\mb z_\delta$ generate the full ring $\CO_D = \End(\CG_F)$. In this case, there is no non-trivial deformation of $\CG_F$ that preserves the full endomorphism ring, so we conclude:
$$
\Int(\delta) = 1.
$$

On the other hand, if $\mb z_\delta^2 \in \varpi_L^2 \CO_L^\times$, then $\Orb(\mb 1, \alpha, s) = -q^s$. If $\alpha$ matches $\delta$, then we have:
$$
1 = \Int(\delta) = -\frac{1}{\ln q} \left. \frac{d}{ds} \right|_{s=0} \Orb(\mb 1, \alpha, s).
$$

Now suppose $\mb z_\delta \in \varpi^2 \CO_D$, which implies $\mb z_\delta \in \varpi^3 \CO_D$. Since $\mb z^2 = (\mb w - \varpi_3)(\mb w - \varpi_3^\sigma)$, the element $\mb w$ must be a uniformizer of $\CO_L$, and thus $1 - \mb z_\delta \in \CO_D^\times$.

By Lemma~\ref{geometrymax}, we have
$$
\Int(\delta_1, \delta_2) = \Int(\delta_{1L}, \widetilde{\delta}).
$$
Let $(\alpha_1, \alpha_2) : (K_0, K_3) \to \Mat_4(F)$ be a pair matching with $(\delta_1, \delta_2)$. Then by Lemma~\ref{revisedmat}, the pair $(\alpha_{1L}, \widetilde{\alpha})$ matches with $(\delta_{1L}, \widetilde{\delta})$.

Applying Theorem~\ref{orbitmat}, we obtain
$$
-\frac{1}{\ln q} \left. \frac{d}{ds} \right|_{s=0} \Orb\left(\mb 1_{\GL_4(\CO_F)}, (\alpha_1, \alpha_2), s \right)
=
-\frac{1}{\ln q} \left. \frac{d}{ds} \right|_{s=0} \Orb\left(\mb 1_{\GL_2(\CO_L)}, (\alpha_{1L}, \widetilde{\alpha}), s \right).
$$

Since $(\alpha_{1L}, \widetilde{\alpha})$ matches with $(\delta_{1L}, \widetilde{\delta})$, and the coquadratic linear AFL holds for $\mb 1_{\GL_2(\CO_L)}$, which was verified by \cite{L22},  we conclude:
$$
-\frac{1}{\ln q} \left. \frac{d}{ds} \right|_{s=0} \Orb\left(\mb 1_{\GL_2(\CO_L)}, (\alpha_{1L}, \widetilde{\alpha}), s \right)
=
\Int(\delta_{1L}, \widetilde{\delta}).
$$

This completes the proof of the theorem.
\end{proof}

%%%%Warning: Qirui is now editing the following files, please do not edit them directly in overleaf because Qirui is using Git...
%\input{QiruiGPTPolishInProgress/OrbitalIntegrals} %%%% Polished complete.
%\input{QiruiGPTPolishInProgress/BFLAnalyticSide} %%% Polished complete
%\input{QiruiGPTPolishInProgress/OrbitalIntegralComputation} %%% Polished complete.
%\input{QiruiGPTPolishInProgress/BFLGeometricSide} %% Polished complete

%%%%%%%%%%%%%%%%%%% The following are history and backups.
%\input{History/OrbitalIntegrals}
%\input{History/BFLAnalyticSide}
%\input{History/OrbitalIntegralComputation}
%\input{History/BFLGeometricSide}

\end{document}